\documentclass[11pt,letterpaper]{amsart}
\usepackage[foot]{amsaddr}

\usepackage{mathtools}
\usepackage{amsmath}
\usepackage[T1]{fontenc}
\usepackage[latin9]{inputenc}
\usepackage{geometry}
\usepackage{wrapfig}
\usepackage{multicol}
\usepackage{graphicx}
\usepackage{soul}
\usepackage{xcolor}
\usepackage{amssymb}
\usepackage{placeins}
\usepackage{bbm}
\usepackage{cite}
\usepackage{caption}
\usepackage{enumerate}
\usepackage{afterpage}
\usepackage{enumitem}
\usepackage{bmpsize}
\usepackage{hyperref}
\usepackage{tabu}
\usepackage{enumitem}
\numberwithin{equation}{section}
\usepackage{stmaryrd}
\usepackage{tikz}
\usetikzlibrary{matrix,graphs,arrows,positioning,calc,decorations.markings,shapes.symbols}
\usepackage{ifthen}
\usetikzlibrary{math}
\usetikzlibrary{matrix,graphs,arrows,positioning,calc,decorations.markings,shapes.symbols}

\makeatletter
\renewcommand{\email}[2][]{%
  \ifx\emails\@empty\relax\else{\g@addto@macro\emails{,\space}}\fi%
  \@ifnotempty{#1}{\g@addto@macro\emails{\textrm{(#1)}\space}}%
  \g@addto@macro\emails{#2}%
}
\makeatother

\newtheorem{theorem}{Theorem}[section]
\newtheorem{lemma}[theorem]{Lemma}
\newtheorem{proposition}[theorem]{Proposition}

{ \theoremstyle{definition}
\newtheorem{definition}[theorem]{Definition}}
{ \theoremstyle{remark}
\newtheorem{remark}[theorem]{Remark}}

\newcommand{\1}{\mathbbm{1}}
\newcommand{\abs}[1]{\lvert #1 \rvert}
\newcommand{\Abs}[1]{\left\lvert #1 \right\rvert}
\newcommand{\Airy}{\mathrm{Airy}}

\newcommand{\calA}{\mathcal{A}}
\newcommand{\calC}{\mathcal{C}}

\newcommand{\calI}{\mathcal{I}}
\newcommand{\diff}{\mathop{}\!\mathrm{d}}
\renewcommand{\diff}{d}

\newcommand{\e}{\varepsilon}
\newcommand{\E}{\mathbb{E}}

\newcommand{\im}{\mathsf{i}}

\newcommand{\ckets}[1]{\left[ #1 \right]}
\newcommand{\ockets}[1]{\mathopen{}\left[ #1 \right]\mathclose{}}

\newcommand{\paren}[1]{\left( #1 \right)}

\newcommand{\ang}[1]{\langle #1 \rangle}
\newcommand{\R}{\mathbb{R}}
\newcommand{\sign}{\operatorname{sign}}
\newcommand{\sub}{\subseteq}

{
\theoremstyle{definition}

}

\renewcommand{\Re}{\operatorname{\mathsf{Re}}}
\renewcommand{\Im}{\operatorname{\mathsf{Im}}}

\newcommand{\refbeth}{\hat{\gamma}}
\newcommand{\Ai}{\mathrm{Ai}}

\newcommand{\crit}{\mathsf{z}}
\newcommand{\GFF}{\mathbf{h}}
\newcommand{\GFFpullback}{\mathbf{h}^{\crit}}
\newcommand{\contour}{\gamma}

\title{From the Airy line ensemble to the Gaussian free field}
\date{\today}
\author{Evgeni Dimitrov} 
\author{Alex Fu}
\author{Zhengye Zhou}

\begin{document}
\begin{abstract}
We study the global fluctuations of the height function associated with the Airy line ensemble. Using its determinantal structure and a steepest-descent analysis of the extended Airy kernel, we prove that, after a suitable rescaling, the centered height function converges to an explicit pullback of the Gaussian free field. The convergence holds in the sense of joint moments of linear statistics against compactly supported continuous test functions.
\end{abstract}
\maketitle

%
%
\section{Introduction and main results}

The Airy line ensemble is a fundamental object in the Kardar--Parisi--Zhang (KPZ) universality class \cite{Corwin12,CH14}. It arises as the edge scaling limit of a wide range of models, including Wigner matrices \cite{Sod15}, lozenge tilings \cite{AH25}, avoiding Brownian bridges (also known as Brownian watermelons) \cite{CH14}, and various integrable models of non-intersecting random walks and last passage percolation \cite{DNV19,DZ26,Zhou26}. As such, it provides a canonical description of edge fluctuations in KPZ-type systems, making the study of its global behavior a natural problem.

At any fixed time, the ensemble restricts to the Airy point process. A fundamental result of Soshnikov \cite{Soshnikov99} established that suitably normalized counting statistics of this point process converge to Gaussian random variables. This raises the question of whether there is a corresponding Gaussian description of the full space-time fluctuations of the ensemble.

A natural object for addressing this question is the associated height function. Rather than focusing on a single counting statistic, one may instead study the random field it generates and investigate its large-scale behavior. The main result of this paper is that, after a suitable rescaling, the fluctuations of this field converge to an explicit pullback of the Gaussian free field (GFF) with Dirichlet boundary conditions. In this sense, our theorem may be viewed as a field-level extension of Soshnikov's central limit theorem, describing the macroscopic fluctuations of the Airy line ensemble.

The appearance of the GFF is consistent with the broader picture that has emerged in the study of random tilings and determinantal systems. In many models whose edge behavior is governed by the Airy line ensemble, the corresponding height functions are known to exhibit GFF fluctuations in the bulk. For lozenge tilings, we refer to \cite{P1,P2} and the book \cite{GorBook}; for domino tilings, see \cite{K01,J05}. From this perspective, it is natural to expect that the Airy line ensemble possesses an underlying Gaussian fluctuation field. Theorem~\ref{Thm.MomentConvergence} confirms this expectation by identifying a pullback of the GFF as the scaling limit of its associated height function.

\medskip

We briefly introduce the Airy line ensemble here and refer the reader to Section \ref{Sec.Preliminary} for further background. The Airy line ensemble is the unique sequence of random continuous functions $\mathcal A=\{\mathcal A_i\}_{i\ge 1}$ satisfying $\mathcal A_i(t)>\mathcal A_{i+1}(t)$ for all \(i\ge 1\) and \(t\in\mathbb R\), whose finite-dimensional distributions are governed by the extended Airy kernel (\ref{Eq.S1AiryKer}). The associated {\em height function} is defined by
\begin{equation}\label{eq: height}
    h(t, x) = \#\{i \geq 1 : \calA_i(t) \geq x\},
\end{equation}
which records the number of Airy lines lying above the level \(x\) at time \(t\). We also let
\[
\mathcal{D}=\{(t,x)\in\mathbb{R}^2:x<0\}
\]
denote the lower half-plane.

With this notation in place, we can state our main result.

\begin{theorem}\label{Thm.MomentConvergence}
    For all $T > 0$, define on $\R^2$ the random functions
    \begin{equation}
        \label{Eq.HeightFunction}
        \begin{aligned}
            h_T(t, x) = h(T^{1/2} t, Tx),\quad
            H_T(t, x) = h_T(t, x) - \E[h_T(t, x)].
        \end{aligned}
    \end{equation}
    For all compactly supported continuous functions $\varphi \in C_c(\R^2)$, define the pairing of $H_T$ and $\varphi$ to be the random variable
    \begin{equation}\label{Eq.Pairing}
        \ang{H_T, \varphi}
        = \int_{\R^2} dt dx H_T(t, x) \varphi(t, x).
    \end{equation}
    Then, for all $m \geq 1$, all $\varphi_1, \ldots, \varphi_m \in C_c(\R^2)$, and all $T_n > 0$ satisfying $T_n \uparrow \infty$, it holds that
    \begin{equation*}
        (\ang{\sqrt{\pi} H_{T_n}, \varphi_1}, \ldots, \ang{\sqrt{\pi} H_{T_n}, \varphi_m})
        \xrightarrow{n \to \infty} (\xi_1, \ldots, \xi_m),
    \end{equation*}
    where the convergence of random vectors is in the sense of joint moments, and $(\xi_1, \ldots, \xi_m)$ is a mean-zero Gaussian random vector with covariance matrix
    \begin{equation}\label{Eq.S1Covariance}
        \E[\xi_j \xi_k]
        =-\frac{1}{2 \pi} \int_{\mathcal{D}} dt_j dx_j \int_{\mathcal{D}} dt_k dx_k \log \left|\frac{t_j + \im \sqrt{-x_j}-t_k - \im \sqrt{-x_k}}{t_j + \im \sqrt{-x_j}-t_k +\im \sqrt{-x_k}}\right| \varphi_j(t_j, x_j) \varphi_k(t_k, x_k).
    \end{equation}
\end{theorem}

\begin{remark}\label{Rem.Thm1}
Lemma \ref{Lem.PairRandomVar} shows that, for every $\varphi \in C_c(\R^2)$, the integral in (\ref{Eq.Pairing}) is almost surely finite and defines a random variable with finite moments. Consequently, $(\ang{\sqrt{\pi} H_{T_n}, \varphi_1}, \ldots, \ang{\sqrt{\pi} H_{T_n}, \varphi_m})$ is a well-defined sequence of random vectors in $\mathbb{R}^m$.
\end{remark}

\begin{remark} In Section~\ref{Section6.2}, we identify the limiting Gaussian random vector $(\xi_1,\ldots,\xi_m)$ with $(\ang{\GFFpullback,\varphi_1},\ldots,\ang{\GFFpullback,\varphi_m})$, where $\GFFpullback$ denotes the $\crit$-pullback of the GFF. The diffeomorphism $\crit$ and the Gaussian field $\GFFpullback$ are introduced in that section.
\end{remark}

The proof of Theorem~\ref{Thm.MomentConvergence} is carried out in Section~\ref{Section6.1}. The starting point is a decomposition of the relevant joint moment into a finite sum of integrals in the variables $t_1,x_1,\ldots,t_m,x_m$. All but one of these integrals are taken over a region in which some variable $x_j$ satisfies $x_j\ge -T_n^{-\e}$. The moment bounds for the counting function of the Airy point process established in Section~\ref{Sec.MomentAsymptotics} imply that these contributions vanish as $T_n\uparrow \infty$. The remaining term involves an integral over the region $x_1,\ldots,x_m\le -T_n^{-\e}$, to which the convergence results of Section~\ref{Sec.CorrelationFunctionConvergence}, obtained through a steepest-descent analysis of the extended Airy kernel, apply directly. These results imply that the term converges to the corresponding joint moment of $(\xi_1,\ldots,\xi_m)$.

From a technical perspective, our approach is closest to those of \cite{BF14, P1}. A key distinction is that the models studied in those works involve finite, discrete height functions, whereas our setting is continuous. As a consequence, a significant part of our analysis is devoted to addressing measurability and integrability issues that arise when defining the pairings in Theorem~\ref{Thm.MomentConvergence} and establishing the existence of their moments. On the other hand, the correlation kernel in our setting admits a considerably more explicit representation than the kernels considered in \cite{BF14, P1}. This additional structure allows us to make several aspects of the asymptotic analysis fully explicit and to provide a detailed treatment of the relevant estimates.

\medskip

We end this section with a brief outline of the paper. Section~\ref{Sec.Preliminary} introduces the Airy line ensemble, establishes the basic determinantal formulas used throughout the paper, and verifies that the pairings appearing in Theorem~\ref{Thm.MomentConvergence} are well defined. Section~\ref{Sec.MomentAsymptotics} derives bounds for the moments and central moments of the number of points in the upper and lower tails of the Airy point process by combining the results of \cite{Soshnikov99} with asymptotic estimates for the Airy function. Section~\ref{sec. estimates} develops the analytic estimates needed for the asymptotic analysis of the extended Airy kernel, including a study of the critical points of the phase function and the associated steep-descent contours. Section~\ref{Sec.CorrelationFunctionConvergence} proves convergence of joint moments of $H_T$ through a careful steepest-descent analysis based on the estimates from Section~\ref{sec. estimates}. Finally, Section~\ref{sec: proof} completes the proof of Theorem~\ref{Thm.MomentConvergence} and explains its connection to the GFF.

%
%
\section{Definitions and preliminary results}\label{Sec.Preliminary} This section collects the definitions and preliminary results required to make precise the statement of Theorem~\ref{Thm.MomentConvergence} and to carry out its proof. We begin by recalling the Airy line ensemble and several of its well-known properties. We then introduce the height observables that appear throughout the paper and establish that their pairings against compactly supported continuous test functions are well defined and possess finite moments of all orders.

%
%
\subsection{The Airy line ensemble} \label{Section2.1} Throughout this section, we freely use terminology pertaining to determinantal point processes, following \cite[Section 2]{Dimitrov26}. We first fix notation and recall the definition of the Airy line ensemble.

\begin{definition}\label{Def.Measures} For a finite set $\mathsf{S} = \{s_1, \ldots, s_m\} \subset \mathbb{R}$, we let $\mu_{\mathsf{S}}$ denote the counting measure on $\mathbb{R}$, defined by $\mu_{\mathsf{S}}(A) = |A \cap \mathsf{S}|$. We also let $\mathrm{Leb}$ denote the usual Lebesgue measure on $\mathbb{R}$ and $\mu_{\mathsf{S}} \times \mathrm{Leb}$ the product measure on $\mathbb{R}^2$.
\end{definition}

\begin{definition}\label{Def.SpaceC} Let $C(\mathbb{N} \times \mathbb{R})$ be the space of continuous functions $f:\mathbb{N} \times \mathbb{R} \rightarrow \mathbb{R}$ with the topology of uniform convergence over compact sets. Here, $\mathbb{N}$ has the discrete topology, $\mathbb{R}$ has the usual topology, and $\mathbb{N} \times \mathbb{R}$ is endowed with the product topology. As shown in \cite[Lemma 2.2]{DEA21}, $C(\mathbb{N} \times \mathbb{R})$ is a Polish space.
\end{definition}

\begin{definition}\label{Def.LE} A {\em line ensemble} $\mathcal{L} = \left(\mathcal{L}(i,t): i \geq 1, t \in \mathbb{R}\right) = \{\mathcal{L}_i\}_{i \geq 1}$ is a random element in $C(\mathbb{N} \times \mathbb{R})$ in the sense of \cite[Section 3]{Billing}. We say that $\mathcal{L}$ is {\em non-intersecting} if almost surely $\mathcal{L}_i(t) > \mathcal{L}_j(t)$ for all $t \in \mathbb{R}$ and $1 \leq i < j$.
\end{definition}

\begin{definition}\label{Def.ExtAiryKer}
For $x_1, x_2, t_1, t_2 \in \mathbb{R}$, we define the {\em extended Airy kernel} by 
\begin{equation}\label{Eq.S1AiryKer}
\begin{split}
K^{\mathrm{Ext.Airy}}(t_1,x_1; t_2,x_2) = \frac{1}{(2\pi \im)^2} \int_{\Im z=\eta} d z \int_{\Im w=\eta'} dw \frac{e^{\im \left(z^3/3 +x_1z + w^3/3 + x_2w\right)}}{\im (z+w) + t_2 - t_1}.
\end{split}
\end{equation}
In (\ref{Eq.S1AiryKer}), $\eta,\eta'>0$ are arbitrary, subject to the condition that $\eta+\eta'+t_1-t_2<0$ whenever $t_2>t_1$.
\end{definition}
\begin{remark}\label{Rem.AiryLE1}
The introduction of the extended Airy kernel is often attributed to \cite{PS02}, where it arises in the context of the polynuclear growth model, although it appeared earlier in \cite{FNH99} and \cite{Mac94}. There are various formulas for the extended Airy kernel, and the one in (\ref{Eq.S1AiryKer}) comes from \cite[Proposition~2.3]{Kurt03}. For later use, we recall the following formula from \cite{PS02}:
 \begin{equation}\label{Eq.altAiryKer}
K^{\mathrm{Ext.Airy}}(t_1,x_1;t_2,x_2)
=
\begin{cases}
\displaystyle
\int_0^\infty
e^{-\lambda (t_1-t_2)}
\Ai(x_1+\lambda)\Ai(x_2+\lambda)\,d\lambda,
& t_1 \ge t_2,
\\[1em]
\displaystyle
-\int_{-\infty}^0
e^{-\lambda (t_1-t_2)}
\Ai(x_1+\lambda)\Ai(x_2+\lambda)\,d\lambda,
& t_1 < t_2,
\end{cases}
\end{equation}
where $\Ai$ is the Airy function.
\end{remark}

With the above notation in place, we can give our formal definition of the Airy line ensemble.
\begin{definition}\label{Def.AiryLE} The {\em Airy line ensemble} $\mathcal{A} = \{\mathcal{A}_i\}_{i \geq 1}$ is a non-intersecting line ensemble in the sense of Definition \ref{Def.LE}. Its law is uniquely characterized by the following property. For any finite set $\mathsf{S} = \{s_1, \ldots, s_m\} \subset \mathbb{R}$ with $s_1 < \cdots < s_m$, the random measure on $\mathbb{R}^2$ given by
\begin{equation}\label{Eq.RMS1}
M^{\mathsf{S}; \mathcal{A}}(A) = \sum_{i \geq 1} \sum_{j = 1}^m {\bf 1} \left\{\left(s_j, \mathcal{A}_{i}(s_j) \right) \in A \right\},
\end{equation}
is a determinantal point process with reference measure $\mu_{\mathsf{S}} \times \mathrm{Leb}$ as in Definition \ref{Def.Measures}, and with correlation kernel given by $K^{\mathrm{Ext.Airy}}$ as in Definition \ref{Def.ExtAiryKer}.
\end{definition}
\begin{remark}\label{Rem.AiryLE2}  The {\em existence} of a line ensemble satisfying the conditions of Definition \ref{Def.AiryLE} was established in \cite{CH14}, while the fact that these conditions {\em uniquely} determine the law of $\mathcal{A}$ is well known and follows, for example, from \cite[Proposition 2.13(3)]{Dimitrov26}, \cite[Corollary 2.20]{Dimitrov26}, and \cite[Lemma 3.1]{DimMat}.
\end{remark}

From either (\ref{Eq.S1AiryKer}) or (\ref{Eq.altAiryKer}), it is clear that $K^{\mathrm{Ext.Airy}}$ depends on $t_1$ and $t_2$ only through $t_1 - t_2$, and so $\mathcal{A}$ is {\em stationary}, i.e. $\left(\mathcal{A}_i(t): i \geq 1, t \in \mathbb{R}\right) \overset{d}{=} \left(\mathcal{A}_{i}(s + t): i \geq 1, t \in \mathbb{R}\right)$ for all $s \in \mathbb{R}$.

We end this section with the following result, which shows that the measure in (\ref{Eq.RMS1}) is a determinantal point process on $\mathbb{R}^2$ with a kernel $\widetilde{K}$ that differs from $K^{\mathrm{Ext.Airy}}$ by a gauge transformation. We introduce this alternative kernel as it is more suitable for asymptotic analysis later in the paper.
\begin{lemma}\label{lemma: Airykernel}
   The random measure $M^{\mathsf{S}; \mathcal{A}}$ in \eqref{Eq.RMS1} is a determinantal point process with reference measure $\mu_{\mathsf{S}} \times \mathrm{Leb}$ as in Definition \ref{Def.Measures}, and with correlation kernel given by 
    \begin{equation}
        \label{Eq.ExtendedAiryKernel}
        \widetilde{K}(t_1, x_1; t_2, x_2) = \frac{1}{(2\pi\im)^2} \int_{\eta + t_1 + \im \mathbb{R}} \diff z \int_{-\eta' + t_2 + \im \mathbb{R} } \diff w \frac{e^{z^3/3 - z^2 t_1 + z(t_1^2 - x_1) - w^3/3 + w^2 t_2 - w(t_2^2 - x_2)}}{z - w}.
    \end{equation}
    In (\ref{Eq.ExtendedAiryKernel}), $\eta,\eta'>0$ are arbitrary, subject to the condition that $\eta+\eta'+t_1-t_2<0$ whenever $t_2>t_1$, and both contours are oriented in the direction of increasing imaginary part.
\end{lemma}
\begin{proof} We start from (\ref{Eq.S1AiryKer}) and perform the change of variables
$z \mapsto -\im z + t_1$ and $w \mapsto \im w + t_2$. Reversing the orientation of the resulting $z$-contour, we get
\begin{equation*}
\widetilde{K}(t_1,x_1;t_2,x_2)
    = \frac{e^{-t_1^3/3+x_1t_1}}
         {e^{-t_2^3/3+x_2t_2}}
    K^{\mathrm{Ext.Airy}}(t_1,x_1;t_2,x_2).
\end{equation*}
As the two kernels are related by a gauge transformation, the lemma follows; see \cite[Proposition~2.13(4)]{Dimitrov26} with
$f(t,x)=e^{-t^3/3+xt}$.
\end{proof}

%
%
\subsection{Height observables}\label{Section2.2} For each $t\in \mathbb{R}$, define the random measure on $\mathbb{R}$ by 
\begin{equation}\label{Eq.AiryPP}
M^{t;\mathcal{A}}(A) = \sum_{i \geq 1}  {\bf 1} \left\{\mathcal{A}_{i}(t) \in A \right\}.
\end{equation}
By Definition \ref{Def.AiryLE} and \cite[Lemma 2.17]{Dimitrov26}, the random measure $M^{t;\mathcal{A}}$ is a determinantal point process with correlation kernel $K^{\mathrm{Airy}}(x_1; x_2) := K^{\mathrm{Ext.Airy}}(t,x_1; t,x_2)$ and reference measure $\mathrm{Leb}$. In particular, $K^{\mathrm{Airy}}$ is the classical {\em Airy kernel}, and the law of $M^{t;\mathcal{A}}$ coincides with that of the {\em Airy point process}. 

Recalling from (\ref{eq: height}) that $h(t, x) = \#\{i \geq 1 : \calA_i(t) \geq x\}$, we can express the height function as
\begin{equation}\label{Eq.HeightToMeasure}
h(t,x) = M^{t; \mathcal{A}}([x, \infty)) = M^{\mathsf{S}; \mathcal{A}}(\{t\} \times [x, \infty)),
\end{equation}
where $\mathsf{S} \subset \mathbb{R}$ is any set that contains $t$, and $M^{\mathsf{S}; \mathcal{A}}$ is as in (\ref{Eq.RMS1}). In particular, $h(t,x)$ is an extended random variable taking values in $\mathbb{Z}_{\geq 0} \cup \{\infty\}$. 

Since $M^{t;\mathcal{A}}$ has the law of the Airy point process, $h(t,x)$ has the same distribution as the number of Airy particles contained in the interval $[x, \infty)$. By \cite[Section 3]{Soshnikov99}, all moments of this counting function are finite. Consequently, $h(t,x) \in \mathbb{Z}_{\geq 0}$ almost surely. Moreover, for every $m \in \mathbb{N}$ and $x \in \mathbb{R}$, there is a constant $C(m,x) \in (0, \infty)$ such that for all $t \in \mathbb{R}$ 
\begin{equation}\label{Eq.MomentBoundS2}
\mathbb{E}\left[h(t,x)^m \right] \leq C(m,x).
\end{equation}

The discussion in the previous paragraph shows that for each $T > 0$ and $t,x \in \mathbb{R}$ the scaled height functions $h_T(t, x)$ and $H_{T}(t,x)$ from (\ref{Eq.HeightFunction}) are random variables that have bounded moments. The next lemma establishes the same for the pairings $\ang{H_T, \varphi}$ in Theorem \ref{Thm.MomentConvergence}.
\begin{lemma}\label{Lem.PairRandomVar}
Fix $T > 0$ and $\varphi \in C_c(\R^2)$. Then $\ang{H_T, \varphi}$ is a random variable with finite moments of all orders.
\end{lemma}
\begin{proof} Let $(\Omega, \mathcal{F}, \mathbb{P})$ be a probability space on which the Airy line ensemble $\mathcal{A} = \{\calA_i\}_{i \geq 1}$ is defined. From \cite[Theorem~1.6]{DV21}, we have
\begin{equation*}
\mathbb{P}\left(\lim_{i \to \infty} \sup_{t \in [-a, a]} \calA_i(t) = -\infty \textup{ for each } a > 0\right) = 1.
\end{equation*}
Consequently, $\mathbb{P}$-almost surely $\mathcal{A}$ takes values in the Borel set
\begin{equation*}
\begin{split}
C_\mathrm{fin} = & \bigcap_{A \in \mathbb{N}} \bigcup_{k \in \mathbb{N}} \left\{\mathcal{C} \in C(\mathbb{N} \times \R): \mathcal{C}_i(t) \geq \mathcal{C}_{i + 1}(t) \mbox{ for all } i \geq 1, t \in \R \mbox{, and}  \sup_{t \in [-A,A]} \mathcal{C}_k(t) < -A  \right\}.
\end{split}
\end{equation*}

We next show that the mapping
\begin{equation*}
\mathcal{C} \mapsto \int_{\R^2} dt dx (\#\{i \geq 1 : \calC_i(t) \geq x\}) \varphi(t, x),
\end{equation*}
is continuous on $C_\mathrm{fin}$. Indeed, suppose that $\mathcal{C}^{(n)} \in C_\mathrm{fin}$ for each $n \geq 1$, and suppose that $\lim_{n \rightarrow \infty} \mathcal{C}^{(n)}= \mathcal{C} \in C_\mathrm{fin}$ uniformly on compact sets. 

If $A$ is large so that $[-A,A]^2$ contains the support of $\varphi$, then we can find $K \in \mathbb{N}$ such that 
$$\sup_{t \in [-A,A]} \mathcal{C}^{(n)}_{K+1}(t) < - A \mbox{ for all $n \in \mathbb{N}$ and } \sup_{t \in [-A,A]} \mathcal{C}_{K+1} (t) < -A.$$
Using the orderedness of $\mathcal{C}_i^{(n)}$, the uniform convergence of $\mathcal{C}_i^{(n)}$ to $\mathcal{C}_i$ on $[-A,A]$ for each $i \in \mathbb{N}$, and the bounded convergence theorem, we conclude
\begin{equation*}
\begin{split}
\lim_{n \rightarrow \infty}\int_{\R^2} dt dx (\#\{i \geq 1 : \calC^{(n)}_i(t) \geq x\}) \varphi(t, x) = \lim_{n \to \infty} \sum_{i=1}^{K} \int_{[-A, A]^2} dt dx \1\{\mathcal{C}^{(n)}_i(t) \geq x\} \varphi(t, x) \\
= \sum_{i=1}^{K} \int_{[-A, A]^2} dt dx \1\{\mathcal{C}_i(t) \geq x\} \varphi(t, x) = \int_{\R^2} dt dx (\#\{i \geq 1 : \calC_i(t) \geq x\}) \varphi(t, x).
\end{split}
\end{equation*}
This proves our desired continuity.

The preceding argument shows that $\ang{h, \varphi}$ is a random variable, being the composition of the measurable map from $(\Omega, \mathcal{F}, \mathbb{P})$ to $C_\mathrm{fin}$ and a continuous map from $C_\mathrm{fin}$ to $\R$. 

Using $h_T(t, x) = h(T^{1/2} t, Tx)$, (\ref{Eq.MomentBoundS2}), and the monotonicity of $h_T(t,x)$ in $x$, we conclude 
$$\mathbb{E}\left[\int_{\R^2} dt dx \left|h_T(t, x) \varphi(t, x)\right|\right] \leq (2A)^2\|\varphi\|_{\infty}  \cdot C(1,-TA).$$ 
Since $\ang{h, \varphi}$ is a random variable for each $\varphi \in C_c(\R^2)$, the same is true for $\ang{h_T, \varphi}$ by scaling. The last inequality shows that $\ang{h_T, \varphi}$ is integrable, and moreover, by Fubini's theorem, it shows that 
\begin{equation}\label{Eq.Center}
\ang{H_T, \varphi} = \int_{\R^2} dt dx \left[h_T(t,x) - \mathbb{E}[h_T(t,x)] \right]\varphi(t, x) = \ang{h_T, \varphi} - \mathbb{E}[\ang{h_T, \varphi}].
\end{equation}
This proves that $\ang{H_T, \varphi}$ is a random variable.

Lastly, by Jensen's inequality, Fubini's theorem, (\ref{Eq.MomentBoundS2}), and the monotonicity of $h_T(t,x)$ in $x$, we have for each $m \geq 1$ that
\begin{equation*}
\begin{split}
&\mathbb{E}\left[|\ang{h_T, \varphi}|^m\right] \leq \mathbb{E}\left[ (2A)^{2m} \|\varphi\|_{\infty}^m \left(\frac{1}{4A^2} \int_{[-A,A]^2} dt dx h_T(t,x) \right)^m \right] \\
& \leq   \mathbb{E}\left[ (2A)^{2m} \|\varphi\|_{\infty}^m \cdot \frac{1}{4A^2} \int_{[-A,A]^2} dt dx h_T(t,x)^m \right] \leq (2A)^{2m}\|\varphi\|_{\infty}^m C(m, -TA).
\end{split}
\end{equation*}
The latter shows that $\ang{h_T, \varphi}$ has finite moments. From (\ref{Eq.Center}), we conclude the same for $\ang{H_T, \varphi}$.
\end{proof}

%
%
\section{Moment bounds}\label{Sec.MomentAsymptotics} In this section, we derive estimates for the moments of the scaled height functions $h_T(t,x)$ and $H_{T}(t,x)$ defined in (\ref{Eq.HeightFunction}). We treat separately the lower-tail regime $x < 0$ in Section \ref{Section3.1} and the upper-tail regime $x > 0$ in Section \ref{Section3.2}. Our arguments combine results from \cite{Soshnikov99} with a detailed analysis of the Airy kernel introduced in Section \ref{Section2.2}. Throughout this section, we use the notation introduced in Section \ref{Sec.Preliminary}.

%
%
\subsection{Lower-tail bounds}\label{Section3.1} As explained in Section~\ref{Section2.2}, $h(t,x)$ has the same distribution as the number of particles of the Airy point process contained in the interval $[x, \infty)$. The latter quantity was thoroughly analyzed in \cite{Soshnikov99} when $x \rightarrow -\infty$, and the results of that paper yield the moment bounds in the next two propositions.

\begin{proposition}\label{Prop.LowerTailCentralMoment}
Fix $m \geq 1$, $A \geq 1$ and $\e \in (0,1)$, and let $H_T$ be as in (\ref{Eq.HeightFunction}). There exists a constant $C_1(m,A,\e) > 0$ such that for all $T \geq e$ and $t \in \mathbb{R}$
\begin{equation}
\label{Eq.LowerTailCentralMoment}
\sup_{x \in [-A, -T^{-\e}]} \E\ockets{\Abs{H_T(t, x)}^m} \le C_1(m,A,\e) \cdot (\log T)^{m/2}.
\end{equation}
\end{proposition}
\begin{proof} From \cite[Theorem 1]{Soshnikov99}, we have that as $T \rightarrow \infty$
\begin{equation}\label{Eq.VarianceExpand}
\mathrm{Var}(h(t,-T)) = \frac{11}{12\pi^2} \cdot \log T + O(1).
\end{equation}
As explained after \cite[Lemma 2]{Soshnikov99}, we can find $C_{\ell} > 0$, depending on $\ell$, such that for all $\ell \geq 2$
\begin{equation}\label{Eq.CumulantBound}
\chi_{\ell}(h(t,-T)) \leq C_{\ell} \cdot \mathrm{Var}(h(t,-T)),
\end{equation}
where $\chi_{\ell}$ is the $\ell$-th order cumulant, see \cite[Chapter 3]{Taqqu}.

From (\ref{Eq.VarianceExpand}), (\ref{Eq.CumulantBound}), and the fact that moments of centered random variables are expressible as finite linear combinations of products of cumulants of order at least $2$, see \cite[Proposition 3.2.1]{Taqqu}, we conclude that there exists a universal $T_1 \geq e$ and constants $C'_{m} > 0$, depending on $m$ alone, such that for $T \geq T_1$ and $m \geq 1$
\begin{equation}\label{Eq.MomentBoundLT}
\mathbb{E}\left[\left(h(t, -T) - \mathbb{E}[h(t,-T)] \right)^{2m}\right] \leq C_{m}' (\log T)^m.
\end{equation}
The inequality in (\ref{Eq.MomentBoundLT}) and the definition of $H_T$ in (\ref{Eq.HeightFunction}) imply (\ref{Eq.LowerTailCentralMoment}) provided that $T^{1-\e} \geq T_1$ and $m$ is even. By using the bound in (\ref{Eq.MomentBoundS2}) and possibly enlarging $C_1(m,A,\e)$, we conclude the inequality for all $T \geq e$ and even $m$. Lastly, the inequality in (\ref{Eq.LowerTailCentralMoment}) for odd $m$ follows from the one for even $m$ and the Cauchy--Schwarz inequality.
\end{proof}

\begin{proposition} \label{Prop.LowerTailMoment} Fix $m \geq 1$, $A \geq 1$, and let $h_T$ be as in (\ref{Eq.HeightFunction}). There exists a constant $C_2(m,A) > 0$ such that for all $T \geq e$ and $t \in \mathbb{R}$
\begin{equation}\label{Eq.LowerTailMoment}
\E[h_T(t, -A)^m] \leq C_2(m,A) \cdot T^{3m/2}.
\end{equation}
In addition, if $\e \in (0,1)$ there exists a constant $C_2'(m,\e) > 0$ such that for all $T \geq e$ and $t \in \mathbb{R}$
\begin{equation}\label{Eq.LowerTailMomentV2}
\E[h_T(t, -T^{-\e})^m] \leq C_2'(m,\e) \cdot T^{3m(1-\e)/2}.
\end{equation}
\end{proposition}
\begin{proof} Put $\mu_T = \mathbb{E}[h(t,-T)]$. As shown above \cite[Theorem 1]{Soshnikov99}, we have as $T \rightarrow \infty$
\begin{equation}\label{Eq.MomentExpandLT}
\mu_T = \frac{2T^{3/2}}{3\pi} + O(1).
\end{equation}
Combining (\ref{Eq.MomentExpandLT}) with the moment bound in (\ref{Eq.MomentBoundS2}), we conclude that there exists a universal constant $C \geq 1$ such that for all $T \geq e$, we have
$\mu_T \leq CT^{3/2}$. Combining the latter with the binomial theorem and the definitions of $h_T, H_T$ in (\ref{Eq.HeightFunction}), we conclude for all $T \geq e$
\begin{equation*}
\begin{split}
&\E[h_T(t, -A)^m] \leq \sum_{j = 0}^m \binom{m}{j} \mu_{TA}^j  \E\ockets{\Abs{H_T(t, -A)}^{m-j}} \\
&\leq \sum_{j = 0}^m \binom{m}{j} C^j (TA)^{3j/2} C_1(m-j,A,1/2) (\log T)^{(m-j)/2},
\end{split}
\end{equation*}
where in the last inequality we used (\ref{Eq.LowerTailCentralMoment}) with $m$ replaced with $m-j$ and $\e = 1/2$, with the convention $C_1(0,A,1/2) = 1$. Since $\log T \leq T^{3}$ for $T \geq e$, the last inequality implies (\ref{Eq.LowerTailMoment}).

We now turn to the proof of (\ref{Eq.LowerTailMomentV2}). From (\ref{Eq.MomentExpandLT}) and the moment bound in (\ref{Eq.MomentBoundS2}), we conclude that there exists a constant $C' \geq 1$, depending on $\e$, such that for all $T \geq e$, we have $\mu_{T^{1-\e}} \leq C' T^{3(1-\e)/2}$. Combining the latter with the binomial theorem and the definitions of $h_T, H_T$ in (\ref{Eq.HeightFunction}), we conclude for all $T \geq e$
\begin{equation*}
\begin{split}
&\E[h_T(t, -T^{-\e})^m] \leq \sum_{j = 0}^m \binom{m}{j} \mu_{T^{1-\e}}^j  \E\ockets{\Abs{H_T(t, -T^{-\e})}^{m-j}} \\
&\leq \sum_{j = 0}^m \binom{m}{j} (C')^j T^{3j(1-\e)/2} C_1(m-j,1,\e) (\log T)^{(m-j)/2},
\end{split}
\end{equation*}
where in the last inequality we used (\ref{Eq.LowerTailCentralMoment}) with $m$ replaced with $m-j$ and $A = 1$, with the convention $C_1(0,1,\e) = 1$. Since $\log T \leq C''T^{3(1-\e)}$ for some $\e$-dependent $C''$ and all $T \geq e$, the last inequality implies (\ref{Eq.LowerTailMomentV2}).
\end{proof}

%
%
\subsection{Upper-tail bounds}\label{Section3.2} In this section, we derive upper-tail bounds for the moments of $H_T(t,x)$ when $x>0$ and $T$ is large. In contrast to the lower-tail regime studied in Section~\ref{Section3.1}, the expected number of Airy particles above the level $Tx$ decays super-exponentially as $T\to\infty$. Consequently, the moments of $H_T(t,x)$ are also negligible on the scales relevant to the proof of Theorem~\ref{Thm.MomentConvergence}. The main result of this section is Proposition~\ref{Prop.UpperTailCentralMoment} below, for which we require the following lemma.

\begin{lemma} \label{Lem.UpperTailMean}
Fix $t \in \mathbb{R}$, $x,T > 0$ with $xT \geq 1$. If $h_T$ is as in (\ref{Eq.HeightFunction}), then for some universal $C_3 > 0$
\begin{equation}\label{Eq.HeightMeanUTB}
\E[h_T(t, x)] \leq C_3 (Tx)^{3/2} e^{-\frac{4}{3} (Tx)^{3/2}}.
\end{equation}
\end{lemma}
\begin{proof}
Using the stationarity of the Airy line ensemble, the identity (\ref{Eq.HeightToMeasure}), the determinantal structure of the Airy point process, and the kernel formula \eqref{Eq.altAiryKer}, we obtain
\begin{equation}\label{Eq.UTMean1}
\E[h_T(t, x)] = \E[h_T(0, x)] = \int_{Tx}^\infty \diff y K^\Airy(y;y) = \int_{Tx}^\infty \diff y \int_y^\infty \diff \lambda \Ai(\lambda)^2.
\end{equation}

By \cite[Chapter 11.1]{Olver74}, for every $\lambda>0$,
\begin{equation}\label{Eq.AiryFunctionBound}
0 < \Ai(\lambda) \leq \frac{e^{-\frac{2}{3}\lambda^{3/2}}}{2\pi^{1/2}\lambda^{1/4}} , \mbox{ and }\abs{\Ai'(\lambda)}\leq \paren{1 + \frac{7}{48 \lambda^{3/2}}} \frac{\lambda^{1/4} e^{-\frac{2}{3}\lambda^{3/2}}}{2 \pi^{1/2} }.
\end{equation}
Integrating (\ref{Eq.UTMean1}) by parts with respect to $y$, and using (\ref{Eq.AiryFunctionBound}) to justify the vanishing of the boundary term at infinity, we obtain
\begin{equation*}
\E[h_T(0, x)] = -Tx \int_{Tx}^\infty \diff \lambda \Ai(\lambda)^2 - \int_{Tx}^\infty \diff y (-y \Ai(y)^2).
\end{equation*}

Using that $\Ai''(\lambda) = \lambda \Ai(\lambda)$, see \cite[Chapter 11.1]{Olver74}, and the preceding identity, we get
\begin{equation*}
\begin{aligned}
\E[h_T(0, x)] & = -Tx \ckets{\lambda \Ai(\lambda)^2 - (\Ai'(\lambda))^2}_{\lambda=Tx}^{\lambda=\infty} + \ckets{\frac{y^2 \Ai(y)^2 - y (\Ai'(y))^2 + \Ai(y) \Ai'(y)}{3}}_{y=Tx}^{y=\infty} \\
& = \frac{2 (Tx)^2 \Ai(Tx)^2 - 2 Tx (\Ai'(Tx))^2 - \Ai(Tx) \Ai'(Tx)}{3}.
\end{aligned}
\end{equation*}
As before, the boundary terms at $\lambda=\infty$ and $y=\infty$ vanish by
\eqref{Eq.AiryFunctionBound}. Applying these same bounds to each term in the final expression above yields (\ref{Eq.HeightMeanUTB}).
\end{proof}

With the preceding result in hand, we are ready to establish a super-exponential upper bound on the moments of $H_T$.
\begin{proposition}\label{Prop.UpperTailCentralMoment}
Fix $t \in \mathbb{R}$, $m \geq 1$, and $x,T > 0$ with $xT \geq 1$. If $H_T$ is as in (\ref{Eq.HeightFunction}), then for some $C_4(m) > 0$,
\begin{equation}\label{Eq.MomentUBUT}
\E[\abs{H_T(t, x)}^m] \leq C_4(m) \cdot (Tx)^{3/2} e^{-\frac{4}{3} (Tx)^{3/2}}.
\end{equation}
\end{proposition}
\begin{proof}
By the stationarity of the Airy line ensemble and \cite[(4.1.3)]{Roman}, we have for $m \geq 1$
\begin{equation}\label{Eq.FallingFactorials}
\E[h_T(t, x)^m] = \E[h_T(0, x)^m] = \sum_{k=1}^m  S(m,k) \E[h_T(0, x)^{\underline{k}}],
\end{equation}
where $S(m,k)$ denotes the Stirling number of the second kind and $a^{\underline{k}} = a (a - 1) \cdots (a - k + 1)$ denotes the falling factorial power.

Using the factorial moment formula for determinantal point processes \cite[(2.13)]{Dimitrov26} and Hadamard's inequality \cite[Corollary 33.2.1.1]{Prasolov}, we obtain
\begin{equation}\label{Eq.Hadamard'sInequality}
\begin{split}
\E[h_T(0, x)^{\underline{k}}]
& = \int_{Tx}^\infty \diff y_1 \cdots \int_{Tx}^\infty \diff y_k \det [\widetilde{K}(0, y_i; 0, y_j)]_{i, j = 1}^k \\
& \leq \int_{Tx}^\infty \diff y_1 \cdots \int_{Tx}^\infty \diff y_k \prod_{j=1}^k \widetilde{K}(0, y_j; 0, y_j) = \E[h_T(0, x)]^k.
\end{split}
\end{equation}
Here we used that the Airy point process is determinantal with correlation kernel $\widetilde{K}(0,x;0,y)$, as follows from \cite[Lemma 2.17]{Dimitrov26} and Lemma~\ref{lemma: Airykernel}.

Combining \eqref{Eq.FallingFactorials}, \eqref{Eq.Hadamard'sInequality}, and Lemma~\ref{Lem.UpperTailMean}, we conclude that there exists a constant $C'_4(m) > 0$ such that
\begin{equation*}
\E[h_T(t, x)^m] \leq C'_4(m) \cdot (Tx)^{3/2} e^{-\frac{4}{3} (Tx)^{3/2}}.
    \end{equation*}
Using the convexity of $x\mapsto |x|^m$, we get
\begin{equation*}
\begin{aligned}
    \E[\abs{H_T(t,x)}^m]
    &=
    2^m
    \E\left[
        \left|
        \frac{h_T(t,x)-\E[h_T(t,x)]}{2}
        \right|^m
    \right] \le
    2^m
    \E\left[
        \frac{|h_T(t,x)|^m+|\E[h_T(t,x)]|^m}{2}
    \right] \\
    &=
    2^{m-1}
    \left(
        \E[h_T(t,x)^m]
        +
        |\E[h_T(t,x)]|^m
    \right) 
    \le
    2^m \E[h_T(t,x)^m],
\end{aligned}
\end{equation*}
where the last inequality follows from Jensen's inequality. The last two inequalities imply (\ref{Eq.MomentUBUT}).
\end{proof}

%
%
\section{Steepest descent preliminaries}\label{sec. estimates}
The proof of Theorem~\ref{Thm.MomentConvergence} relies on a steepest descent analysis, which is carried out in Section~\ref{Sec.CorrelationFunctionConvergence}. In preparation for that analysis, we introduce the relevant phase function and contours, establish decay estimates along the latter, and derive several bounds for the associated contour integrals that will be used later in the proof.

%
%
\subsection{Phase function and contours}\label{Section4.1} Our main object of interest in this section is the following function, which depends on $z \in \mathbb{C}$ and $t,x \in \mathbb{R}$:
\begin{equation}\label{Eq.DefS}
S(z; t, x) = \frac{z^3}{3} - z^2 t + z(t^2 - x).
\end{equation}
The function $S$ appears in the exponent of the kernel formula (\ref{Eq.ExtendedAiryKernel}) and plays a central role in the steepest descent analysis. In what follows, we analyze some of its properties.

Note that $S'(z;t,x) = (z-t)^2 - x$, which for $x < 0$ has the complex conjugate roots
\begin{equation}\label{Eq.RootsOfS}
\crit(t,x) = \crit^+(t,x) = t + \im \sqrt{-x} \mbox{ and } \crit^-(t, x)  = \overline{\crit(t, x)} = t - \im \sqrt{-x}.
\end{equation}
We next introduce certain contours that pass through these critical points.
\begin{definition}\label{def: contour} Fix $t \in \mathbb{R}$, $x < 0$, $\eta > 0$ and let $\crit^{\pm} = \crit^{\pm}(t,x)$ be as in (\ref{Eq.RootsOfS}), and $\bar{\eta} = \min(\eta, \sqrt{-x})$. Given this data, we define the contour $\contour = \contour(t,x,\eta)$ by
\begin{equation}\label{Eq.ContourBeth}
\begin{aligned}
\contour = & \{\crit^+  - \sqrt{3}\eta + \im u : u \in (\eta, \infty)\} \cup \{\crit^+  - \sqrt{3}u + \im u : u \in [-\bar{\eta}, \eta]\} \\
& \cup \{ t + \sqrt{3}\bar{\eta} + \im u : u \in [-\sqrt{-x} + \bar{\eta}, \sqrt{-x} - \bar{\eta}]\}  \\
& \cup \{\crit^-  - \sqrt{3}u - \im u : u \in [-\bar{\eta}, \eta]\} \cup \{\crit^-  - \sqrt{3}\eta - \im u : u \in (\eta, \infty)\},
\end{aligned}
\end{equation}
oriented in the direction of increasing imaginary part; see Figure~\ref{fig:braids}. We also define the following subcontours of $\contour$, equipped with the same orientation
\begin{equation*}
\begin{aligned}
&\contour ^\pm = \contour  \cap \{\pm \Im z > 0\}, \hspace{2mm}  \contour^{\pm,0} = \{\crit^\pm  - \sqrt{3}u \pm \im u : u \in [-\bar{\eta}, \eta]\}, \hspace{2mm} \contour^{\pm,1}  = \contour^\pm  \setminus \contour^{\pm,0}.
\end{aligned}
\end{equation*}
We denote by $\refbeth $ the reflection of $\contour $ across the line $t  + \im \mathbb{R}$, and we define $\refbeth^\pm $, $\refbeth^{\pm,0} $, and $\refbeth^{\pm,1} $ analogously.
\end{definition}
\begin{remark}
In words, $\contour^{+,0} $ and $\contour^{-,0}$ are exactly the two diagonal line segments of $\contour$. In addition, $\contour^{+,1}$ consists of one vertical line segment connecting $t + \sqrt{3}\bar{\eta}$ and $t + \sqrt{3}\bar{\eta} + \im (\sqrt{-x} - \bar{\eta})$, and the ray from $t - \sqrt{3}\eta + \im (\sqrt{-x} + \eta)$ and going vertically up. An analogous statement holds for $\contour^{-,1}$. Note that when $\eta > \sqrt{-x}$, the line segments in $\contour^{\pm,1}$ degenerate to single points.
\end{remark}

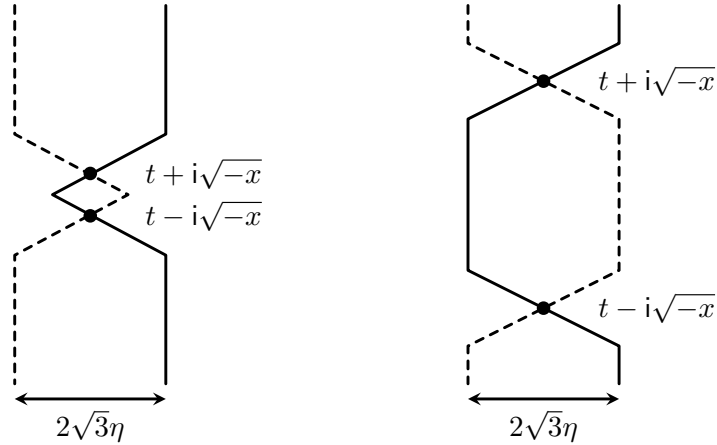
\begin{figure}[ht]
        \centering
        \begin{tikzpicture}[line width=1.1pt, line cap=round, line join=round]
            \pgfmathsetmacro{\aA}{0.5}
            \pgfmathsetmacro{\aB}{1.5}
            
            \newcommand{\twistpair}[4]{%
            \pgfmathsetmacro{\w}{1}
            \pgfmathsetmacro{\yt}{#2}
            \pgfmathsetmacro{\yb}{#3}
            
            \pgfmathsetmacro{\yA}{#2-\aA}
            \pgfmathsetmacro{\yB}{#2-\aB}
            \pgfmathsetmacro{\yC}{#3+\aB}
            \pgfmathsetmacro{\yD}{#3+\aA}
            
            \pgfmathsetmacro{\ycUp}{(#2-\aA + #2-\aB)/2}
            \pgfmathsetmacro{\ycDn}{(#3+\aB + #3+\aA)/2}
            
            \draw[black,dashed] (#1-\w,\yt)
                  -- (#1-\w,\yA)
                  -- (#1+\w,\yB)
                  -- (#1+\w,\yC)
                  -- (#1-\w,\yD)
                  -- (#1-\w,\yb);
            
            \draw (#1+\w,\yt)
                  -- (#1+\w,\yA)
                  -- (#1-\w,\yB)
                  -- (#1-\w,\yC)
                  -- (#1+\w,\yD)
                  -- (#1+\w,\yb);

            \filldraw[black] (#1,\ycUp) circle (2pt);
            \filldraw[black] (#1,\ycDn) circle (2pt);
            
            \node[inner sep=1pt] at (#1+1.5,\ycUp) {$t +\im\sqrt{-x}$};
            \node[inner sep=1pt] at (#1+1.5,\ycDn) {$t -\im\sqrt{-x}$};
            \draw [stealth-stealth] (#1+\w,\yb-0.2) -- (#1-\w,\yb-0.2) node [midway, below] {$2\sqrt{3}\eta$};

            }

            \newcommand{\othertwistpair}[4]{%
            \pgfmathsetmacro{\otheraA}{1.5}
            \pgfmathsetmacro{\otheraB}{\otheraA + 0.7}
            
            \pgfmathsetmacro{\w}{1}
            \pgfmathsetmacro{\otherw}{\w - 0.5}
            \pgfmathsetmacro{\yt}{#2}
            \pgfmathsetmacro{\yb}{#3}
            
            \pgfmathsetmacro{\yA}{#2-\otheraA -0.2}
            \pgfmathsetmacro{\yB}{#2-\otheraB}
            \pgfmathsetmacro{\yC}{#3+\otheraB}
            \pgfmathsetmacro{\yD}{#3+\otheraA + 0.2}
            
            \pgfmathsetmacro{\ycUp}{0.28}
            \pgfmathsetmacro{\ycDn}{-0.28}
            
            \draw[black,dashed] (#1-\w,\yt)
                  -- (#1-\w,\yA)
                  -- (#1+\otherw,0)
                  -- (#1-\w,\yD)
                  -- (#1-\w,\yb);
            
            \draw (#1+\w,\yt)
                  -- (#1+\w,\yA)
                  -- (#1-\otherw,0)
                  -- (#1+\w,\yD)
                  -- (#1+\w,\yb);

            \filldraw[black] (#1,\ycUp) circle (2pt);
            \filldraw[black] (#1,\ycDn) circle (2pt);
            
            \node[inner sep=1pt] at (#1+1.5,\ycUp) {$t+\im\sqrt{-x}$};
            \node[inner sep=1pt] at (#1+1.5,\ycDn) {$t-\im\sqrt{-x}$};

            \draw [stealth-stealth] (#1+\w,\yb-0.2) -- (#1-\w,\yb-0.2) node [midway, below] {$2\sqrt{3}\eta$};
            }
            
            \othertwistpair{-3}{2.5}{-2.5}{1}
            \twistpair{3}{2.5}{-2.5}{2}
        
        \end{tikzpicture}
        \caption{The contour $\contour$ (dashed) and its reflection $\refbeth$ (solid). The left side depicts $\contour$ and $\refbeth$ when $\eta \geq \sqrt{-x}$, and the right side depicts these contours when $\eta < \sqrt{-x}$.} 
        \label{fig:braids}
    \end{figure}

The next lemma summarizes some well-separatedness estimates for the contours in Definition \ref{def: contour}.
\begin{lemma}\label{Lem.Separated} Fix $t \in \mathbb{R}$, $x < 0$, $\eta > 0$, and assume the notation in Definition \ref{def: contour}. If $(z,w) \in \refbeth \times \contour^{\pm, 1}$ or $(z,w) \in \refbeth^{\pm,1} \times \contour$, then $|z - w| \geq \eta$.
\end{lemma}
\begin{proof} We prove the statement when $(z,w) \in \refbeth \times \contour^{\pm, 1}$, and note that it implies the one when $(z,w) \in \refbeth^{\pm, 1} \times \contour$ by reflection symmetry across the vertical line through $t$.

Suppose first that $\eta < \sqrt{-x}$. Looking at the right side of Figure \ref{fig:braids}, we know that $w$ belongs to one of the three dashed vertical lines. There are four points closest to $\contour$; their distance is the same and equal to the height of a right triangle with sides $2\eta, 2\sqrt{3}\eta, 4\eta$. Consequently, $|z-w| \geq \sqrt{3}\eta$.

Suppose next that $\eta \geq \sqrt{-x}$. Looking at the left side of Figure \ref{fig:braids}, we know that $w$ belongs to one of the two dashed vertical lines. If $\Re(z) \geq t$, then $|z-w| \geq |\Re(z-w)| \geq\sqrt{3} \eta$. If $\Re(z) \leq t$, then $|z-w| \geq |\Im(z-w)| \geq \min(|\Im(w) - \sqrt{-x}|, |\Im(w) + \sqrt{-x}|) \geq \eta$.
\end{proof}

The following lemma estimates the real part of $S$ along the contours in Definition \ref{def: contour}.
\begin{lemma}\label{Lem.Derivatives} Fix $t \in \mathbb{R}$, $x < 0$, $\eta > 0$. Let $S(z;t,x)$ be as in (\ref{Eq.DefS}) and assume the notation in Definition \ref{def: contour}. Then the following statements hold.
\begin{enumerate}
\item[(a)] If $z^{\pm} = \crit^\pm +\sqrt{3}u \pm \im u \in \refbeth^{\pm,0}$, then 
\begin{equation}\label{eq: upper1}
\Re \left[S(z^{\pm}; t, x)- S(\crit^\pm;t, x)\right] = -2 \sqrt{3} \cdot \sqrt{-x} u^2.
\end{equation}
\item[(b)] If $z^{\pm} = \crit^\pm + \sqrt{3}\eta \pm \im u \in \refbeth^{\pm,1}$ and $u \geq \eta$, then
\begin{equation}\label{eq: upper2}
\Re \left[S(z^{\pm}; t, x)- S(\crit^\pm;t, x)\right] \leq -\sqrt{-x} u\eta.
\end{equation}
\item[(c)] If $\eta< \sqrt{-x}$, $z^{\pm}=\crit^\pm - \sqrt{3}\eta \mp \im u \in \refbeth^{\pm,1}$ and $\eta \leq u \leq \sqrt{-x}$, then
\begin{equation}\label{eq: upper3}
\Re \left[S(z^{\pm}; t, x)- S(\crit^\pm;t, x)\right] \leq -\sqrt{-x}\eta^2.
\end{equation}
\end{enumerate}
Moreover, the same inequalities hold with $\refbeth^\pm$ replaced by $\contour^\pm$ and $S$ replaced by $-S$.
\end{lemma}
\begin{proof} Since $S(\bar z;t,x)=\overline{S(z;t,x)}$ for all $z\in\mathbb C$, the statements for the ``minus'' sign follow from those for the ``plus'' sign by complex conjugation. Thus, it suffices to prove the latter.

A direct computation yields
\begin{equation*}
S(\crit^+; t, x) = \frac{t^3 - 3tx - 4\im (-x)^{3/2}}{3},\quad S'(\crit^+; t,x) = 0, \quad S''(\crit^+; t, x) =  2 \im \sqrt{-x}.
\end{equation*}
Consequently, the Taylor expansion of $S$ around $\crit^+$ is
\begin{equation*}
S(z; t, x) = \frac{t^3}{3} - tx - \im \frac{4(-x)^{3/2}}{3} + \im \sqrt{-x} (z - \crit^+)^2 + \frac{1}{3} (z - \crit^+)^3.
\end{equation*}
Setting $A = \Re(z - \crit^+)$ and $B = \Im(z - \crit^+)$, we obtain
\begin{equation}\label{Eq.RealPartMain}
\Re \left[S(z; t, x)- S(\crit^+;t, x)\right] = \paren{- 2 \sqrt{-x} B + \frac{A^2 - 3B^2}{3}} A.
\end{equation}

Equation (\ref{eq: upper1}) follows from (\ref{Eq.RealPartMain}) with $A = \sqrt{3}u$, $B = u$. If $z = \crit^+ + \sqrt{3}\eta + \im u$ with $u \geq \eta$, then (\ref{Eq.RealPartMain}) with $A = \sqrt{3}\eta$ and $B = u$ implies
$$\Re \left[S(z; t, x)- S(\crit^+;t, x)\right] = \paren{- 2 \sqrt{-x} u + \frac{3\eta^2 - 3u^2}{3}} \sqrt{3}\eta \leq -2\sqrt{3} \cdot \sqrt{-x}u\eta,$$
which is stronger than (\ref{eq: upper2}). Similarly, if $z = \crit^+ - \sqrt{3}\eta - \im u$, $\eta< \sqrt{-x}$ and $\eta \leq u \leq \sqrt{-x}$, then from (\ref{Eq.RealPartMain}) with $A = -\sqrt{3}\eta$ and $B = -u$, we get
$$\Re \left[S(z; t, x)- S(\crit^+;t, x)\right] = -\paren{ 2 \sqrt{-x} u + \eta^2 -u^2} \sqrt{3}\eta \leq -\paren{ \sqrt{-x} u + \eta^2} \sqrt{3}\eta \leq -\sqrt{3} \sqrt{-x} \eta^2,$$
which is stronger than (\ref{eq: upper3}). In the last inequalities we used $\eta \leq u \leq \sqrt{-x}$. 

The final statement in the lemma follows by reflection symmetry, and we omit the details.
\end{proof}

%
%
\subsection{Integral bounds}\label{Section4.2} In this section, we use the inequalities in Lemma \ref{Lem.Derivatives} to derive appropriate bounds for several contour integrals that arise later in our arguments. The latter are summarized in the following lemma.

\begin{lemma}Fix $t \in \mathbb{R}$, $x < 0$, $\eta > 0$, and $T > 0$. Let $S(z;t,x)$ be as in (\ref{Eq.DefS}) and assume the notation in Definition \ref{def: contour}. With $|dz|$ and $|dw|$ denoting integration with respect to arc-length, the following inequalities all hold. 
\begin{equation}\label{Eq.IntegrandUpperBound3}
\begin{aligned}
& \int_{\refbeth^{\pm,1}} |\diff z |\left|e^{T^{3/2} \left(S(z; t, x)- S(\crit^\pm;t, x)\right)}\right|
\le \left(\sqrt{-x} + \frac{1}{T^{3/2}\sqrt{-x} \eta} \right)e^{-T^{3/2}\sqrt{-x}\eta^2},\\
& \int_{\contour^{\pm,1}} |\diff w| \left|e^{-T^{3/2} \left(S(w; t, x)- S(\crit^\pm;t, x)\right)}\right|
\le \left(\sqrt{-x} + \frac{1}{T^{3/2}\sqrt{-x} \eta} \right)e^{-T^{3/2}\sqrt{-x}\eta^2},
\end{aligned}
\end{equation}

\begin{equation}\label{Eq.IntegrandUpperBound4}
\begin{aligned}
& \int_{\refbeth^{\pm,0}} |\diff z|  \left|e^{T^{3/2} \left(S(z; t, x)- S(\crit^\pm;t, x)\right)}\right|\le \frac{4}{T^{3/4}(-x)^{1/4}},\\
& \int_{\contour^{\pm,0}} |\diff w|  \left|e^{-T^{3/2} \left(S(w; t, x)- S(\crit^\pm;t, x)\right)}\right|\le \frac{4}{T^{3/4}(-x)^{1/4}}.
\end{aligned}
\end{equation}

\begin{equation}\label{Eq.IntegrandUpperBound5}
\begin{aligned}
&\int_{\refbeth^{\pm,0}} \abs{\diff z} \int_{\contour^{\pm,0}} \abs{\diff w}\, \frac{\left|e^{T^{3/2}\left(S(z;t,x)-S(w;t,x)\right)}\right|}{\abs{z-w}} \leq \frac{32}{T^{3/4} (-x)^{1/4}}, \\
&\int_{\refbeth^{\pm,0}} \abs{\diff z} \int_{\contour^{\mp,0}} \abs{\diff w}\, \frac{\left|e^{T^{3/2}\left(S(z;t,x)-S(w;t,x)\right)}\right|}{\abs{z-w}} \leq \frac{32}{T^{3/4} (-x)^{1/4}}.
\end{aligned}
\end{equation}
\end{lemma}
\begin{proof}
By combining \eqref{eq: upper2} and \eqref{eq: upper3}, we conclude
\begin{equation*}
\begin{aligned}
&\int_{\refbeth^{\pm,1}} |\diff z |\left|e^{T^{3/2} \left(S(z; t, x)- S(\crit^\pm;t, x)\right)}\right|
 \leq \1_{\eta < \sqrt{-x}} \int_{\eta}^{\sqrt{-x}} \diff u e^{-T^{3/2}\sqrt{-x}\eta^2} + \int_\eta^\infty \diff u e^{-T^{3/2} \sqrt{-x} \eta u} \\
& = \1_{\eta < \sqrt{-x}} \cdot (\sqrt{-x}-\eta)  e^{-T^{3/2}\sqrt{-x}\eta^2} + \frac{e^{-T^{3/2}\sqrt{-x} \eta^2}}{T^{3/2}\sqrt{-x} \eta}  \leq \left(\sqrt{-x} + \frac{1}{T^{3/2}\sqrt{-x} \eta}\right)e^{-T^{3/2}\sqrt{-x}\eta^2},
\end{aligned}
\end{equation*}
which establishes the first line in (\ref{Eq.IntegrandUpperBound3}). The second one is verified analogously.

From \eqref{eq: upper1} we obtain
\begin{equation*}
\begin{aligned}
&\int_{\refbeth^{\pm,0}} |\diff z|  \left|e^{T^{3/2} \left(S(z; t, x)- S(\crit^\pm;t, x)\right)}\right|
 =  2 \int_{-\bar{\eta}}^{\eta} du e^{-2\sqrt{3} T^{3/2} \sqrt{-x} u^2}\\
& \leq 2 \int_{-\infty}^{\infty} du e^{-2\sqrt{3} T^{3/2} \sqrt{-x} u^2}
= \frac{2\pi^{1/2}}{T^{3/4}2^{1/2}3^{1/4}(-x)^{1/4}} \leq \frac{4}{T^{3/4}(-x)^{1/4}},
\end{aligned}
\end{equation*}
which establishes the first line in (\ref{Eq.IntegrandUpperBound4}). The second one is verified analogously.

By parametrizing $z = \crit^\pm +\sqrt{3}u \pm \im u$ and $w = \crit^\pm +\sqrt{3}v \mp \im v$, we get
\begin{equation*}
\begin{aligned}
&\int_{\refbeth^{\pm,0}} \abs{\diff z}
\int_{\contour^{\pm,0}} \abs{\diff w}\,
\frac{\left|e^{T^{3/2}\left(S(z;t,x)-S(w;t,x)\right)}\right|}{\abs{z-w}} \leq 
\int_{-\infty}^{\infty} 2\,\diff u
\int_{-\infty}^{\infty} 2\,\diff v\,
\frac{e^{-T^{3/2}\sqrt{-x} [u^2 + v^2]}}{\sqrt{4u^2+4v^2 - 2\sqrt{3}uv + 2uv}} \\
& \leq 4\int_{-\infty}^{\infty} \diff u
\int_{-\infty}^{\infty} \diff v\,
\frac{e^{-T^{3/2}\sqrt{-x} [u^2 + v^2]}}{\sqrt{u^2+v^2}}  = 8\pi \int_0^{\infty}dr e^{-T^{3/2}\sqrt{-x}r^2} =  \frac{4 \pi^{3/2}}{T^{3/4} (-x)^{1/4}} \leq  \frac{32}{T^{3/4} (-x)^{1/4}},
\end{aligned}
\end{equation*}
where in the last line we changed to polar coordinates. This verifies the first line in (\ref{Eq.IntegrandUpperBound5}). 

Since $S(\bar z;t,x)=\overline{S(z;t,x)}$ for all $z\in\mathbb C$, and $|z-w| \geq |z - \bar{w}|$ for $z \in \contour^{\pm,0}$ and $w \in \refbeth^{\mp,0}$, we get
$$\int_{\refbeth^{\pm,0}} \abs{\diff z} \int_{\contour^{\mp,0}} \abs{\diff w}\, \frac{\left|e^{T^{3/2}\left(S(z;t,x)-S(w;t,x)\right)}\right|}{\abs{z-w}}\leq \int_{\refbeth^{\pm,0}} \abs{\diff z} \int_{\contour^{\pm,0}} \abs{\diff w}\, \frac{\left|e^{T^{3/2}\left(S(z;t,x)-S(w;t,x)\right)}\right|}{\abs{z-w}},$$
and so the second line in (\ref{Eq.IntegrandUpperBound5}) follows from the first. 
\end{proof}

%
%
\section{Convergence of joint moments}\label{Sec.CorrelationFunctionConvergence} In this section, we establish the convergence of the joint moments of the centered height function $H_T$. The main result is Proposition~\ref{Prop.CorrelationFunctionConvergence}, which serves as a key ingredient in the proof of Theorem~\ref{Thm.MomentConvergence}.

%
%
\subsection{Main proposition}\label{Section5.1} We begin by introducing some notation in the following definition and then state the main result of the section.

\begin{definition}\label{Def.ITs}
Fix $\ell \geq 1$. In what follows, indices are understood modulo $\ell$; in particular, we identify $\ell+1$ with $1$. Define
\begin{equation}\label{Eq.I_T}
\calI_T(t_1, x_1; \ldots; t_\ell, x_\ell) = \int_{Tx_1}^\infty \diff y_1 \cdots \int_{Tx_\ell}^\infty \diff y_\ell \prod_{j=1}^\ell \widetilde{K}(T^{1/2} t_j, y_j; T^{1/2} t_{j+1}, y_{j+1}),
\end{equation}
where $\widetilde{K}$ is as in \eqref{Eq.ExtendedAiryKernel}. We also define $\calI(t_1,x_1) = 0$ and for $\ell \geq 2$
\begin{equation}\label{Eq.I}
\calI(t_1, x_1; \ldots; t_\ell, x_\ell) = \frac{1}{(2\pi\im)^\ell} \int_{\overline{\crit(t_1, x_1)}}^{\crit(t_1, x_1)} \diff z_1 \cdots \int_{\overline{\crit(t_\ell, x_\ell)}}^{\crit(t_\ell, x_\ell)} \diff z_\ell \prod_{j=1}^\ell \frac{1}{z_j - z_{j+1}},
\end{equation}
where we recall from \eqref{Eq.RootsOfS} that $\crit(t,x) = t + \im \sqrt{-x}$. Lastly, define 
\begin{equation*}
P(t_1, x_1; \ldots; t_\ell, x_\ell) = \sum_{\substack{\sigma\in S_{\ell}\\ \sigma\textup{ has no fixed points}}}
\sign(\sigma) \prod_{\substack{(i_1\,\cdots\, i_\ell)\\ \textup{cycle of }\sigma}} \calI(t_{i_1},x_{i_1};\ldots;t_{i_\ell},x_{i_\ell}),
\end{equation*}
where the product is taken over cycles in the cycle decomposition of $\sigma$.
\end{definition}

\begin{proposition}\label{Prop.CorrelationFunctionConvergence} Fix $m \geq 2$,  $\e \in (0,1)$,  $\delta \in \left(0,\tfrac{1}{2m}\right)$, and $A > 0$. Let $H_T$ be as in (\ref{Eq.HeightFunction}) and assume the same notation as in Definition \ref{Def.ITs}. Then, as $T \to \infty$,
\begin{equation*}
\E\bigl[H_T(t_1,x_1)\cdots H_T(t_m,x_m)\bigr] \longrightarrow P(t_1, x_1; \ldots; t_m, x_m)
\end{equation*}
uniformly for
\begin{equation*}
    (t_1,x_1;\ldots;t_m,x_m)
    \in
    \bigl([-A,A]^2\cap\{(t,x):x\le -T^{-\e}\}\bigr)^m
    \cap
    \Bigl\{
        \min_{i\neq j}|t_i-t_j|>T^{-\delta}
    \Bigr\}.
\end{equation*}
\end{proposition}

We end this subsection with the following lemma. It provides a representation for the joint moments of $H_T$ that will be used in the proof of Proposition~\ref{Prop.CorrelationFunctionConvergence}.

\begin{lemma}\label{Lem.JointMomentCycleDecomp1}
Let $H_T$ be as in (\ref{Eq.HeightFunction}) and assume the same notation as in Definition \ref{Def.ITs}. Fix $m \geq 1$,  $x_1, \ldots, x_m \in \mathbb{R}$ and pairwise distinct $t_1,\ldots,t_m \in \mathbb{R}$. Then for all $T > 0$, we have
\begin{equation}\label{Eq.MomentFormula}
\E\bigl[H_T(t_1,x_1)\cdots H_T(t_m,x_m)\bigr] =
\sum_{\substack{\sigma\in S_m\\ \sigma\textup{ has no fixed points}}}
\sign(\sigma) \prod_{\substack{(i_1\,\cdots\, i_\ell)\\ \textup{cycle of }\sigma}} \calI_T(t_{i_1},x_{i_1};\ldots;t_{i_\ell},x_{i_\ell}),
\end{equation}
where the product is taken over cycles in the cycle decomposition of $\sigma$.
\end{lemma}
\begin{proof}
The result is standard; nevertheless, we record its proof for completeness and to keep the exposition self-contained. By the definition of $H_T$ in \eqref{Eq.HeightFunction},
\begin{equation}\label{Eq.JointMomDecomp1}
\E[H_T(t_1, x_1) \cdots H_T(t_m, x_m)] = \sum_{S \sub \{1, \ldots, m\}} \E\ockets{\prod_{k \in S} h(T^{1/2}t_k, Tx_k)} \prod_{k \notin S} \left(-\E[h(T^{1/2}t_k, Tx_k)]\right).
\end{equation}
    
Combining (\ref{Eq.HeightToMeasure}) with the joint moment formula for determinantal point processes \cite[(2.13)]{Dimitrov26}, and Lemma \ref{lemma: Airykernel}, we conclude
\begin{equation}\label{Eq.JointMomDecomp2}
\begin{aligned}
& \E\ockets{\prod_{k \in S} h(T^{1/2}t_k, Tx_k)} \prod_{k \notin S} \left(-\E[h(T^{1/2}t_k, Tx_k)]\right)  = \int_{Tx_1}^\infty \diff y_1 \cdots \int_{Tx_m}^\infty \diff y_m \\
& \det[\widetilde{K}(T^{1/2} t_i, y_i; T^{1/2} t_j, y_j)]_{i, j \in S} \cdot \prod_{k \notin S} \left( -\widetilde{K}(T^{1/2} t_k, y_k; T^{1/2} t_k, y_k) \right).
\end{aligned}
\end{equation}
We mention that in applying \cite[(2.13)]{Dimitrov26}, we used that the sets $\{T^{1/2}t_k\} \times [Tx_k, \infty)$ are pairwise disjoint by assumption.

If $A$ is an $m\times m$ matrix and for each subset $S\subseteq\{1,\ldots,m\}$, we let $A[S]$ denote the principal submatrix of $A$ indexed by $S$, we have the trivial identity
\begin{equation*}
\det\bigl(A+\operatorname{diag}(b_1,\ldots,b_m)\bigr) = \sum_{S\subseteq \{1,\ldots,m\}} \det(A[S])\prod_{k\notin S} b_k.
\end{equation*}
Applying the latter with 
\begin{equation*}
A_{i,j} = \widetilde K(T^{1/2}t_i,y_i;T^{1/2}t_j,y_j), \qquad b_i = -\widetilde K(T^{1/2}t_i,y_i;T^{1/2}t_i,y_i),
\end{equation*}
and utilizing (\ref{Eq.JointMomDecomp1}) and (\ref{Eq.JointMomDecomp2}), we get
\begin{equation*}
\begin{aligned}
\E[H_T(t_1,x_1)\cdots H_T(t_m,x_m)] &=\int_{Tx_1}^\infty \diff y_1 \cdots \int_{Tx_m}^\infty \diff y_m 
\det\!\Bigl[(1-\delta_{i,j}) \widetilde K(T^{1/2}t_i,y_i;T^{1/2}t_j,y_j)\Bigr]_{i,j=1}^m.
\end{aligned}
\end{equation*}
The last equality and the Leibniz formula for the determinant imply (\ref{Eq.MomentFormula}).
\end{proof}

%
%
\subsection{Proof of Proposition \ref{Prop.CorrelationFunctionConvergence}}\label{Section5.2}
In view of Lemma \ref{Lem.JointMomentCycleDecomp1}, it suffices to show for each $\ell \in \{1, \ldots, m\}$
\begin{equation}\label{Eq.SDRed1}
\lim_{T \rightarrow \infty}\calI_T(t_1,x_1;\ldots;t_\ell,x_\ell) \longrightarrow \calI(t_1,x_1;\ldots;t_\ell,x_\ell),
\end{equation}
uniformly for
\begin{equation}\label{Eq.UniformConvSetL}
(t_1,x_1;\ldots;t_\ell,x_\ell) \in \bigl([-A,A]^2\cap\{(t,x):x\le -T^{-\e}\}\bigr)^\ell \cap \Bigl\{\min_{i\neq j}|t_i-t_j|>T^{-\delta}\Bigr\}.
\end{equation}

Without loss of generality, we may assume that
\[
t_1 < t_2 < \cdots < t_{\ell}.
\]
We organize the proof of (\ref{Eq.SDRed1}) into three steps. In Step~1, we derive an alternative representation of $\calI_T$ that involves the phase function $S$ from Section \ref{Section4.1}. In Step~2, we deform the contours in the representation of $\calI_T$ obtained in Step~1 to the steep-descent contours from Definition~\ref{def: contour}. The contour deformation crosses a collection of simple poles, and Cauchy's residue theorem yields the decomposition of $\calI_T$ as a sum of contour integrals in \eqref{Eq.I_TDissected}. In Step~3, we apply the estimates for the integrands along $\refbeth^\pm$ and $\contour^\pm$ obtained in Section \ref{Section4.2} and the dominated convergence theorem to conclude (\ref{Eq.SDRed1}). Throughout the proof, all indices are taken modulo $\ell$.
\newline
    
\noindent
\textbf{Step~1.} We choose
\begin{equation}\label{Eq.Eta}
\eta=\eta' := \min_{i\neq j}\frac{|t_i-t_j|}{8} > \frac{T^{-\delta}}{8}
\end{equation}
in the definition of $\widetilde{K}$ in \eqref{Eq.ExtendedAiryKernel}. With this choice of $\eta$ and $\eta'$, it follows that
\begin{equation*}
    T^{1/2}\eta + T^{1/2}t_i
    <
    T^{1/2}(-\eta) + T^{1/2}t_j
\end{equation*}
whenever $t_i<t_j$.
 
From \eqref{Eq.ExtendedAiryKernel} and (\ref{Eq.I_T}), we conclude
\begin{equation*}
\begin{aligned}
\calI_T(t_1, x_1; \ldots; t_\ell, x_\ell)
= \frac{1}{(2\pi\im)^{2\ell}} \int_{Tx_1}^\infty \diff y_1 \cdots \int_{Tx_\ell}^\infty \diff y_\ell \prod_{j=1}^\ell \int_{T^{1/2} \eta + T^{1/2} t_j + \im \mathbb{R}} \diff z_j \int_{-T^{1/2} \eta + T^{1/2} t_{j+1} + \im \mathbb{R}} \diff w_j \\
\frac{e^{z_j^3/3 - z_j^2 T^{1/2} t_j + z_j (T t_j^2 - y_j) - w_j^3/3 + w_j^2 T^{1/2} t_{j+1} - w_j (T t_{j+1}^2 - y_{j+1})}}{z_j - w_j}.
\end{aligned}
\end{equation*}

We change the order of integration using Fubini's theorem, and evaluate the resulting inner integrals, which are convergent because $\Re(-z_j + w_{j-1}) = -2T^{1/2} \eta < 0$, obtaining
\begin{equation*}
\begin{aligned}
\calI_T(t_1, x_1; \ldots; t_\ell, x_\ell) = &\frac{1}{(2\pi\im)^{2\ell}} \prod_{j=1}^\ell \int_{T^{1/2} \eta + T^{1/2} t_j + \im \mathbb{R}} \diff z_j \int_{T^{1/2} (-\eta) + T^{1/2} t_{j+1} + \im \mathbb{R}} \diff w_j  \\
& \prod_{j=1}^\ell \frac{e^{z_j^3/3 - z_j^2 T^{1/2} t_j + z_j (T t_j^2 - Tx_j) - w_j^3/3 + w_j^2 T^{1/2} t_{j+1} - w_j (T t_{j+1}^2 - Tx_{j+1})}}{(z_j - w_j) (z_j - w_{j-1})}.
\end{aligned}
\end{equation*}
    
Lastly, we make the change of variables $z_j \mapsto z_j/T^{1/2}, w_j \mapsto w_j/T^{1/2}$ for all $j$, and use the definition of $S(z; t, x)$ from \eqref{Eq.DefS} to obtain
\begin{equation}\label{eq: ITstep1}
\begin{aligned}
\calI_T(t_1, x_1; \ldots; t_\ell, x_\ell) = &\frac{1}{(2\pi\im)^{2\ell}} \prod_{j=1}^\ell \int_{ \eta + t_j + \im \mathbb{R}} \diff z_j \int_{-\eta + t_{j+1} + \im \mathbb{R}} \diff w_j \\
&  \prod_{j=1}^\ell \frac{e^{T^{3/2} \left(S(z_j; t_j, x_j) - S(w_j; t_{j+1}, x_{j+1})\right)}}{(z_j - w_j) (z_j - w_{j-1})}.
\end{aligned}
\end{equation}

\noindent
\textbf{Step~2.} Let $\contour_j = \contour(t_j, x_j, \eta)$, $\refbeth_j = \refbeth(t_j, x_j, \eta)$ be as in Definition \ref{def: contour}, where $\eta$ is as in (\ref{Eq.Eta}). We further set $\crit_j = \crit(t_j,x_j)$ and $\crit^{\pm}_j = \crit^{\pm}(t_j,x_j)$ as in (\ref{Eq.RootsOfS}). 

Starting from (\ref{eq: ITstep1}), we deform the vertical contour for $z_j$ to $\refbeth_j$, and we deform the vertical contour for $w_j$ to $\contour_{j+1}$. During this deformation, the only residues encountered are those coming from the simple poles of the factors $z_j-w_{j-1}$. Note that by our choice of $\eta$ in \eqref{Eq.Eta}, the poles corresponding to $z_j-w_j$ are not crossed: indeed, when $z_j\in \refbeth_j$ and $w_j\in \contour_{j+1}$,
\begin{equation*}
\abs{z_j-w_j} \geq \abs{\Re(z_j-w_j)} \geq \abs{t_j-t_{j+1}}-\abs{\Re(z_j)-t_j}-\abs{\Re(w_j)-t_{j+1}} \geq (8 - 2\sqrt{3})\eta >0.
\end{equation*}
By Cauchy's residue theorem, after the contour deformation \eqref{eq: ITstep1} becomes
\begin{equation*}
\begin{aligned}
\calI_T(t_1, x_1; \ldots; t_\ell, x_\ell) = & \sum_{J \subseteq \{1, \ldots, \ell\}} 
\frac{1}{(2\pi\im)^{2\ell-\abs{J}}}\prod_{k \notin J} \int_{\refbeth_k} \diff z_k \int_{\contour_k} \diff w_{k-1} \prod_{j \in J}\int_{\crit^-_j}^{\crit^+_j} \diff z_j \\
& \prod_{k \notin J} \frac{e^{T^{3/2}(S(z_k;t_k,x_k)-S(w_{k-1};t_k,x_k))}}
{(z_{k}-w_{k})(z_k-w_{k-1})}
\cdot \prod_{j \in J}  \,\frac{1}{z_{j-1}-z_j}.
\end{aligned}
\end{equation*}
In the last equation we have set $w_k = z_{k+1}$ when $k + 1 \in J$, and the $z_j$ integrals are over the vertical segments that pass through $t_j$ and connect $\crit_j^-$ to $\crit_j^+$ with upward orientation.
  
Furthermore, we write each $\refbeth_k$ as the disjoint (up to endpoints) union of $\refbeth^{+,0}_k$, $\refbeth^{+,1}_k$, $\refbeth^{-,0}_k$, $\refbeth^{-,1}_k$ (as in Definition~\ref{def: contour}), and likewise for $\contour_k$, which yields
\begin{equation}\label{Eq.I_TDissected}
\begin{aligned}
&\calI_T(t_1, x_1; \ldots; t_\ell, x_\ell)
= \sum_{J \subseteq \{1, \ldots, \ell\}} \frac{1}{(2\pi\im)^{2\ell-\abs{J}}}  \prod_{k \notin J} \sum_{\substack{\e_k \in \{+, -\}, \alpha_k \in \{0, 1\} \\ \hat{\e}_k \in \{+, -\}, \hat{\alpha}_k \in \{0, 1\}}} \int_{\refbeth_k^{\hat{\e}_k, \hat{\alpha}_k}} \diff z_k \int_{\contour_k^{\e_k, \alpha_k}} \diff w_{k-1}\\
&  \prod_{j \in J}\int_{\crit^-_j}^{\crit^+_j} \diff z_j  \prod_{k \notin J} \frac{e^{T^{3/2}(S(z_k;t_k,x_k)-S(w_{k-1};t_k,x_k))}}
{(z_{k}-w_{k})(z_k-w_{k-1})}
\cdot \prod_{j \in J}  \,\frac{1}{z_{j-1}-z_j}.
\end{aligned}
\end{equation}
\newline
    
\noindent
\textbf{Step~3.} Within \eqref{Eq.I_TDissected}, we note that the term corresponding to $J=\{1,\ldots,\ell\}$ is exactly equal to $\calI(t_1,x_1;\ldots;t_\ell,x_\ell)$, as defined in \eqref{Eq.I}. Consequently, to prove (\ref{Eq.SDRed1}), it suffices to show that terms corresponding to $J \subsetneq \{1, \ldots, \ell\}$ converge to zero uniformly over the set (\ref{Eq.UniformConvSetL}). 

In what follows we fix $J \subsetneq \{1, \ldots, \ell\}$. For each $k\notin J$, there are $16$ possible choices of
\begin{equation*}
(\e_k,\alpha_k,\hat{\e}_k,\hat{\alpha}_k)
\in \{+,-\}\times\{0,1\}\times\{+,-\}\times\{0,1\}.
\end{equation*}
Fix one such choice for every $k\notin J$, and call the corresponding term on the right side of (\ref{Eq.I_TDissected}) by $\mathsf{D}_T(J; \vec{\e}, \vec{\alpha})$. What remains is to show that 
\begin{equation}\label{Eq.SDRed2}
\lim_{T \rightarrow \infty}\mathsf{D}_T(J; \vec{\e}, \vec{\alpha}) = 0.
\end{equation}

By the definition of the contours and our choice of $\eta$ in \eqref{Eq.Eta}, we have for all $j \in J$ and $k \not \in J$
\begin{equation}\label{Eq.IntegrandUpperBound1V1}
\frac{1}{\abs{z_{j-1} - z_j}}, \frac{1}{\abs{z_{k} - w_{k}}}\leq \frac{1}{\eta}.
\end{equation}
In addition, if $(z_k,w_{k-1})\in (\refbeth^{\pm,1}_k\times\contour_k) \cup (\refbeth_k\times\contour^{\pm,1}_k)$, we have from Lemma \ref{Lem.Separated}
\begin{equation}\label{Eq.IntegrandUpperBound1V2}
\frac{1}{\abs{z_{k} - w_{k-1}}}\leq \frac{1}{\eta}.
\end{equation}

In order to bound the integrals in $\mathsf{D}_T(J; \vec{\e}, \vec{\alpha})$, it is convenient to define
\begin{equation*} J':=\{k\notin J:\alpha_k=1 \textup{ or } \hat{\alpha}_k=1\}, \qquad
J'':= \{k\notin J:\alpha_k=\hat{\alpha}_k=0\}.
\end{equation*}
In words, for $k\in J''$ both contours $\refbeth_k^{\hat{\e}_k, \hat{\alpha}_k}$ and $\contour_k^{\e_k, \alpha_k}$ are diagonal pieces, whereas for $k\in J'$, at least one contour segment is a vertical ray (when $\eta \geq \sqrt{-x_k}$) or a union of a vertical ray and a segment (when $\eta < \sqrt{-x_k}$), see Figure \ref{fig:braids}.

Using the above notation, (\ref{Eq.I_TDissected}), (\ref{Eq.IntegrandUpperBound1V1}), (\ref{Eq.IntegrandUpperBound1V2}) and the identity
\begin{equation*}
\int_{\crit^-_j}^{\crit^+_j} |\diff z_j|= 2 \sqrt{-x_j} ,
\end{equation*}
we get 
\begin{equation*}
\begin{split}
&\left|\mathsf{D}_T(J; \vec{\e}, \vec{\alpha})\right| \leq \frac{1}{(2\pi)^{2\ell-\abs{J}}}  \prod_{k \in J'} \int_{\refbeth_k^{\hat{\e}_k, \hat{\alpha}_k}} |\diff z_k| \int_{\contour_k^{\e_k, \alpha_k}} |\diff w_{k-1}| \frac{\left|e^{T^{3/2}(S(z_k;t_k,x_k)-S(w_{k-1};t_k,x_k))}\right|}
{\eta^2}\\
& \times  \prod_{k \in J''} \int_{\refbeth_k^{\hat{\e}_k, \hat{\alpha}_k}} |\diff z_k| \int_{\contour_k^{\e_k, \alpha_k}} |\diff w_{k-1}| \prod_{k \notin J} \frac{\left|e^{T^{3/2}(S(z_k;t_k,x_k)-S(w_{k-1};t_k,x_k))}\right|}
{|z_{k}-w_{k-1}| \eta}
\cdot \prod_{j \in J}  \frac{2 \sqrt{-x_j}}{\eta}. 
\end{split}
\end{equation*}
Combining the last inequality with \eqref{Eq.IntegrandUpperBound3}, \eqref{Eq.IntegrandUpperBound4}, and \eqref{Eq.IntegrandUpperBound5}, we conclude
\begin{equation}\label{Eq.MasterBoundD}
\begin{split}
&\left|\mathsf{D}_T(J; \vec{\e}, \vec{\alpha})\right| \leq \frac{1}{(2\pi)^{2\ell-\abs{J}}} \prod_{\{k \in J' : \alpha_k = 1\} \uplus \{k \in J' : \hat{\alpha}_k = 1\}} \frac{\left(\sqrt{-x_k} + \frac{1}{T^{3/2}\sqrt{-x_k} \eta} \right)e^{-T^{3/2}\sqrt{-x_k}\eta^2}}{\eta}\\
& \times \prod_{\{k \in J' : \alpha_k = 0\} \uplus \{k \in J' : \hat{\alpha}_k = 0\}} \frac{4}{T^{3/4} (-x_k)^{1/4} \eta}  \cdot  \prod_{k \in J''} \frac{32}{T^{3/4}(-x_k)^{1/4} \eta}
\cdot \prod_{j \in J}  \frac{2 \sqrt{-x_j}}{\eta}.
\end{split}
\end{equation}
where $\uplus$ denotes multiset union (with multiplicity).

Recall from \eqref{Eq.Eta} that $\eta > T^{-\delta}/8$, and recall the prevailing assumption from (\ref{Eq.UniformConvSetL}) that $T^{-\e} \leq -x_1, \ldots, -x_\ell \leq A$ for $T$ large enough and some fixed $A > 0$. If $\{k \in J' : \alpha_k = 1\}$ or $\{k \in J' : \hat{\alpha}_k = 1\}$ is nonempty, then (\ref{Eq.MasterBoundD}) implies for some $C_1 > 0$, depending on $\ell, A, \delta$, and $\e$, and all large $T$
\begin{equation}\label{Eq.BoundD1}
\begin{split}
&\left|\mathsf{D}_T(J; \vec{\e}, \vec{\alpha})\right| \leq \exp\left(C_1 \log T - T^{3/2 - \e/2 - 2\delta}/64\right).
\end{split}
\end{equation}
If $\{k \in J' : \alpha_k = 1\} = \{k \in J' : \hat{\alpha}_k = 1\} = \emptyset$, then (\ref{Eq.MasterBoundD}) implies for some $C_2 > 0$, depending on $\ell, A, \delta$, and $\e$, and all large $T$
\begin{equation}\label{Eq.BoundD2}
\begin{split}
&\left|\mathsf{D}_T(J; \vec{\e}, \vec{\alpha})\right| \leq C_2 \cdot T^{|J|\delta} \cdot T^{(2\abs{J'} + \abs{J''}) (-\frac{3}{4} + \delta + \frac{\e}{4})}.
\end{split}
\end{equation}
The inequalities (\ref{Eq.BoundD1}) and (\ref{Eq.BoundD2}) readily imply (\ref{Eq.SDRed2}), once we recall that $\delta < \frac{1}{2m}$, $\e < 1$, $|J| \leq m-1$, and $2\abs{J'} + \abs{J''} \geq 1$.

%
%
\section{Convergence of pairings}\label{sec: proof} In Section~\ref{Section6.1} we apply the results of Sections~\ref{Sec.MomentAsymptotics} and~\ref{Sec.CorrelationFunctionConvergence} to complete the proof of Theorem~\ref{Thm.MomentConvergence}. In Section~\ref{Section6.2} we briefly review the Gaussian free field and explain its connection to the limiting Gaussian random vector $(\xi_1,\ldots,\xi_m)$ appearing in Theorem~\ref{Thm.MomentConvergence}.

%
%
\subsection{Proof of Theorem~\ref{Thm.MomentConvergence}}\label{Section6.1}
For clarity, we split the proof of the theorem into two steps. In Step 1, we express $\E[\ang{H_T, \varphi_1} \cdots \ang{H_T, \varphi_m}]$ as a sum of $3^m$ terms, see (\ref{Eq.JointMomentAsSum}), and assume that all but one of the terms vanish as $T \rightarrow \infty$, see (\ref{Eq.VanishS6}). Subsequently, we use Proposition \ref{Prop.CorrelationFunctionConvergence} to show that the surviving term converges and identify its limit with the moment $\E[\xi_1\cdots\xi_m]$. In Step 2, we verify (\ref{Eq.VanishS6}) using the moment bounds from Section~\ref{Sec.MomentAsymptotics}.

\medskip

{\bf \raggedleft Step 1.} We fix $m \geq 1$ and $\varphi_1, \ldots, \varphi_m \in C_c^\infty(\R^2)$. In addition, we suppose that $A, B \geq 1$ are sufficiently large so that $\|\varphi_1\|_{\infty}, \ldots, \|\varphi_m\|_{\infty} \leq B$, and $S = [-A, A]^2$ contains the supports of $\varphi_1, \ldots, \varphi_m$. Lastly, we fix $\e \in (\frac{3}{5},1)$, and decompose $S = S^{(1)}_T \sqcup S^{(2)}_T \sqcup S^{(3)}_T$, where
\begin{equation*}
\begin{aligned}
S^{(1)}_T  = \{(t, x) \in S : T^{-\e} \leq x\}, \hspace{2mm} S^{(2)}_T = \{(t, x) \in S :|x| < T^{-\e}\}, \hspace{2mm}S^{(3)}_T  = \{(t, x) \in S : x \leq -T^{-\e}\}.
\end{aligned}
\end{equation*}

By Fubini's theorem, whose application is justified by the moment bounds in (\ref{Eq.MomentBoundS2}) and H{\"o}lder's inequality, we have
\begin{equation}\label{Eq.JointMomentAsSum}
\begin{aligned}
&\E[\ang{H_T, \varphi_1} \cdots \ang{H_T, \varphi_m}] = \sum_{i_1, \ldots, i_m \in \{1, 2, 3\}}R_T(i_1, \ldots, i_m), \mbox{ where }  \\
&R_T(i_1, \ldots, i_m) = \left(\prod_{j = 1}^m\int_{S^{(i_j)}_T } \diff t_j \diff x_j \right) 
\E[H_T(t_1, x_1) \cdots H_T(t_m, x_m)] \prod_{j=1}^m \varphi_j(t_j, x_j).
\end{aligned}
\end{equation}

We claim that unless $i_j = 3$ for all $j \in \{1, \ldots, m\}$, we have
\begin{equation}\label{Eq.VanishS6}
\begin{aligned}
\lim_{T \rightarrow \infty}R_T(i_1, \ldots, i_m) = 0.
\end{aligned}
\end{equation}
We establish (\ref{Eq.VanishS6}) in the next step. Here, we assume its validity and conclude the proof of the theorem.\\

Fix $\delta \in (0, \frac{1}{2m})$ and let $\Delta_T = \{(t_1, x_1; \ldots; t_m, x_m) : \min_{i \neq j} \abs{t_i - t_j} > T^{-\delta}\}$. By H\"{o}lder's inequality and the stationarity of the Airy line ensemble in the first inequality below, and then by Proposition~\ref{Prop.LowerTailCentralMoment} in the second inequality,
\begin{equation}\label{Eq.JointMomentPiece3Close}
\begin{aligned}
& \limsup_{T \to \infty} \Abs{\int_{(S^{(3)}_T)^m \setminus \Delta_T} \diff t_1 \diff x_1 \cdots \diff t_m \diff x_m \E[H_T(t_1, x_1) \cdots H_T(t_m, x_m)] \prod_{j=1}^m \varphi_j(t_j, x_j)} \\
& \leq \limsup_{T \to \infty} B^m \cdot \mathrm{Leb}\left((S^{(3)}_T)^m \setminus \Delta_T\right) \cdot  \sup_{x \in [-A, -T^{-\e}]}\E[\abs{H_T(0, x)}^m] \\
& \leq \limsup_{T \to \infty} B^m \cdot \binom{m}{2} 2 T^{-\delta} (2A)^{2m-1} \cdot C_1(m,A,\e)(\log T)^{m/2} = 0.
\end{aligned}
\end{equation}

We next observe that if $\calI$ and $P$ are as in Definition \ref{Def.ITs}, then
\begin{equation}\label{Eq.RewriteP}
\begin{aligned}
P(t_1, x_1; \ldots; t_m, x_m) = \sum_{\substack{\sigma \in S_m \\ \sigma \textup{ is a pairing}}} \sign(\sigma) \prod_{\substack{(i_1 \ i_2) \\ \textup{ cycle of } \sigma}} \calI(t_{i_1}, x_{i_1}; t_{i_2}, x_{i_2}),
\end{aligned}
\end{equation}
where the sum is over all permutations that are products of two-cycles, and in particular equals zero if $m$ is odd. The latter follows immediately from the formula for $P$ and the identity
\begin{equation*}
\sum_{\substack{\sigma \in S_m \\ \sigma \textup{ has no fixed points} \\ \sigma \textup{ contains a cycle of length at least } 3}} \sign(\sigma) \prod_{\substack{(i_1 \ \cdots \ i_\ell) \\ \textup{ cycle of } \sigma}} \calI(t_{i_1}, x_{i_1}; \ldots; t_{i_\ell}, x_{i_\ell}) = 0,
\end{equation*}
which holds by \cite[Lemma~7.3]{Kenyon08}.

Combining (\ref{Eq.RewriteP}) with the identity
\begin{equation*}
\calI(t_1, x_1; t_2, x_2) = \frac{1}{(2\pi\im)^2} \int_{\overline{\crit(t_1,x_1)}}^{\crit(t_1,x_1)} \diff z_1 \int_{\overline{\crit(t_2,x_2)}}^{\crit(t_2,x_2)} \diff z_2 \frac{1}{(z_1 - z_2)(z_2 - z_1)} = \frac{-1}{2\pi^2} \log\Abs{\frac{\crit(t_1,x_1) - \overline{\crit(t_2,x_2)}}{\crit(t_1,x_1) - \crit(t_2,x_2)}},
\end{equation*}
which follows from the definition of $\calI$ in (\ref{Eq.I}), we conclude
\begin{equation}\label{Eq.RewriteP2}
\begin{aligned}
P(t_1, x_1; \ldots; t_m, x_m) = \sum_{\substack{\sigma \in S_m \\ \sigma \textup{ is a pairing}}} \prod_{\substack{(i_1 \ i_2) \\ \textup{ cycle of } \sigma}} \frac{1}{2\pi^2} \log\Abs{\frac{\crit(t_{i_1},x_{i_1}) - \overline{\crit(t_{i_2},x_{i_2})}}{\crit(t_{i_1},x_{i_1}) - \crit(t_{i_2},x_{i_2})}}.
\end{aligned}
\end{equation}

Combining Proposition \ref{Prop.CorrelationFunctionConvergence} with the integrability of the right side of (\ref{Eq.RewriteP2}) over $[-A,A]^{2m}$, and the dominated convergence theorem, we conclude for $m \geq 2$
\begin{equation}\label{Eq.Conv3Piece}
\begin{aligned}
& \lim_{T \to \infty} \int_{(S^{(3)}_T)^m \cap \Delta_T} \diff t_1 \diff x_1 \cdots \diff t_m \diff x_m \E[H_T(t_1, x_1) \cdots H_T(t_m, x_m)] \prod_{j=1}^m \varphi_j(t_j, x_j) \\
& = \left(\prod_{j = 1}^m\int_{[-A,A]} dt_j \int_{[-A,0]} dx_j \right) P(t_1, x_1; \ldots; t_m, x_m) \prod_{j=1}^m \varphi_j(t_j, x_j).
\end{aligned}
\end{equation}
When $m = 1$, both sides of (\ref{Eq.Conv3Piece}) are equal to $0$ for all $T > 0$ and so the equality holds for all $m \geq 1$.

Combining (\ref{Eq.JointMomentAsSum}), (\ref{Eq.VanishS6}), (\ref{Eq.JointMomentPiece3Close}), and (\ref{Eq.Conv3Piece}), we conclude 
\begin{equation}\label{Eq.FinalMoment}
\begin{aligned}
&\lim_{T \rightarrow \infty} \E[\ang{\sqrt{\pi} H_T, \varphi_1} \cdots \ang{\sqrt{\pi} H_T, \varphi_m}] \\
&= \pi^{m/2}\left(\prod_{j = 1}^m\int_{[-A,A]} dt_j \int_{[-A,0]} dx_j \right) P(t_1, x_1; \ldots; t_m, x_m) \prod_{j=1}^m \varphi_j(t_j, x_j).
\end{aligned}
\end{equation}
Using (\ref{Eq.RewriteP2}), the fact that $\varphi_j$ are supported in $[-A,A]$, and $\mathcal{D}=\{(t,x)\in\mathbb{R}^2:x<0\}$, we see that the second line of (\ref{Eq.FinalMoment}) equals
 $$\sum_{\substack{\sigma \in S_m \\ \sigma \textup{ is a pairing}}} \prod_{\substack{(i_1 \ i_2) \\ \textup{ cycle of } \sigma}} \frac{1}{2\pi} \int_{\mathcal{D}} dt_{i_1} dx_{i_1} \int_{\mathcal{D}} dt_{i_2} dx_{i_2}  \log\Abs{\frac{\crit(t_{i_1},x_{i_1}) - \overline{\crit(t_{i_2},x_{i_2})}}{\crit(t_{i_1},x_{i_1}) - \crit(t_{i_2},x_{i_2})}} \varphi_{i_1}(t_{i_1},x_{i_1}) \varphi_{i_2}(t_{i_2},x_{i_2}).$$
From {\em Wick's formula}, see \cite[(3.2.21)]{Taqqu}, and (\ref{Eq.S1Covariance}), the last expression is precisely $\mathbb{E}[\xi_1 \cdots \xi_m]$. In other words, (\ref{Eq.FinalMoment}) shows that 
$$\lim_{T \rightarrow \infty} \E[\ang{\sqrt{\pi} H_T, \varphi_1} \cdots \ang{\sqrt{\pi} H_T, \varphi_m}] = \mathbb{E}[\xi_1 \cdots \xi_m],$$
which is the moment convergence statement we were after.\\

{\bf \raggedleft Step 2.} In this step we establish (\ref{Eq.VanishS6}). Suppose first that at least one of $i_1,\ldots,i_m$ is equal to $1$. By symmetry, we may assume without loss of generality that $i_1=1$. We observe the following tower of inequalities.
\begin{equation} \label{Eq.JointMomentPiece1}
\begin{aligned}
& \limsup_{T \to \infty}
        \Abs{
        \int_{S^{(i_1)}_T \times \cdots \times S^{(i_m)}_T}
        \diff t_1 \diff x_1 \cdots \diff t_m \diff x_m\,
        \E[H_T(t_1,x_1)\cdots H_T(t_m,x_m)]
        \prod_{j=1}^m \varphi_j(t_j,x_j)
        } \\
        & \leq \limsup_{T \to \infty} B^m \int_{S^{(1)}_T} \diff t_1 \diff x_1\, \E[\abs{H_T(0,x_1)}^m]^{\frac{1}{m}}
        \prod_{j=2}^m \int_{S^{(i_j)}_T} \diff t_j \diff x_j\, \E[\abs{H_T(0,x_j)}^m]^{\frac{1}{m}} \\
        & \leq \limsup_{T \to \infty} B^m (4A^2)^{m-1} 2^{\frac{m^2 - 1}{m}}\E[h_T(0,-A)^m]^{\frac{m-1}{m}} \int_{S^{(1)}_T} \diff t_1 \diff x_1\, C_4(m)^{1/m} (Tx_1)^{\frac{3}{2m}} e^{-\frac{4}{3m} (Tx_1)^{3/2}}  \\
        & \leq \limsup_{T \to \infty} B^m(4A^2)^{m} \cdot 2^{\frac{m^2 - 1}{m}} C_2(m,A)^{\frac{m-1}{m}}T^{\frac{3(m-1)}{2}} \cdot
         T^{\frac{3(1-\e)}{2m}} e^{-\frac{4}{3m} T^{3(1-\e)/2}} =0.
    \end{aligned}
\end{equation}
Let us elaborate on (\ref{Eq.JointMomentPiece1}) briefly. In going from the first to the second line, we used the triangle inequality to put the absolute value inside the integral, the inequality $\|\varphi_j\| \leq B$, H\"older's inequality and the stationarity of the Airy line ensemble. In going from the second to the third line we used Proposition \ref{Prop.UpperTailCentralMoment}, the fact that $\mathrm{Leb}(S_T^{(i)}) \leq 4A^2$, and the inequality 
$$\E[\abs{H_T(0,x_j)}^m] \leq 2^m \E[h_T(0,x_j)^m] + 2^m  \E[h_T(0,x_j)]^m \leq 2^{m+1}\E[h_T(0,-A)^m],$$
which follows from $|a-b|^m \leq 2^m a^m + 2^m b^m$ for $a,b \geq 0$, Jensen's inequality and the monotonicity of $h_T(t,x)$ in $x$. In going from the third to the fourth line we used Proposition \ref{Prop.LowerTailMoment}, the fact that $x \geq T^{-\e}$ for $x \in S_T^{(1)}$ and $\mathrm{Leb}(S_T^{(1)}) \leq 4A^2$. The last equality used $\e \in (\frac{3}{5},1)$.

We next consider the case in which all $i_1, \ldots, i_m \in \{2, 3\}$ and at least one of $i_1, \ldots, i_m$ is equal to $2$. Let $I_2 = |\{j: i_j = 2\}|$ and $I_3 = |\{j: i_j = 3\}|$. We observe the following tower of inequalities.
\begin{equation}
\label{Eq.JointMomentPiece2}
\begin{aligned}
&\limsup_{T \to \infty} \Abs{\int_{S^{(i_1)}_T \times \cdots \times S^{(i_m)}_T} \diff t_1 \diff x_1 \cdots \diff t_m \diff x_m \E[H_T(t_1, x_1) \cdots H_T(t_m, x_m)] \prod_{j=1}^m \varphi_j(t_j, x_j)} \\
&\leq \limsup_{T \to \infty} B^m \prod_{j : i_j = 2}  \int_{S^{(i_j)}_T} \diff t_j \diff x_j\, \E[\abs{H_T(0,x_j)}^m]^{\frac{1}{m}}
        \prod_{j : i_j = 3} \int_{S^{(i_j)}_T} \diff t_j \diff x_j\, \E[\abs{H_T(0,x_j)}^m]^{\frac{1}{m}}  \\ 
&\leq \limsup_{T \to \infty} B^m  \cdot  (4A T^{-\e})^{I_2} \left( 2^{m+1} \E[h_T(0,-T^{-\e})^m]\right)^{\frac{I_2}{m}}
        \cdot (4A^2)^{I_3}\sup_{x_j \in [-A, -T^{-\e}]} \hspace{-2mm} \E[\abs{H_T(0,x_j)}^m]^{\frac{I_3}{m}}  \\ 
&\leq \limsup_{T \to \infty} B^m  \cdot (4A T^{-\e})^{I_2} \left(  2^{m+1} C_2'(m,\e) T^{\frac{3m(1-\e)}{2}}\right)^{\frac{I_2}{m}}
\cdot (4A^2)^{I_3}C_1(m,A,\e)^{\frac{I_3}{m}}(\log T)^{\frac{I_3}{2}}  \\ 
& \leq \limsup_{T \to \infty} C(m,A,B,\e) \cdot T^{-\e I_2} \cdot T^{\frac{3I_2(1-\e)}{2}} \cdot (\log T)^{\frac{I_3}{2}} = 0.
\end{aligned}
\end{equation}
Let us elaborate on (\ref{Eq.JointMomentPiece2}) briefly. In going from the first to the second line, we used the triangle inequality to put the absolute value inside the integral, the inequality $\|\varphi_j\| \leq B$, H\"older's inequality and the stationarity of the Airy line ensemble. In going from the second to the third line, we used that $\mathrm{Leb}(S^{(2)}_T) \leq 4AT^{-\e}$, $\mathrm{Leb}(S^{(3)}_T) \leq 4A^2$, and the inequality
$$\E[\abs{H_T(0,x_j)}^m] \leq 2^m \E[h_T(0,x_j)^m] + 2^m  \E[h_T(0,x_j)]^m \leq 2^{m+1}\E[h_T(0,-T^{-\e})^m],$$
which follows from $|a-b|^m \leq 2^m a^m + 2^m b^m$ for $a,b \geq 0$, Jensen's inequality and the monotonicity of $h_T(t,x)$ in $x$. In going from the third to the fourth line we used Propositions \ref{Prop.LowerTailCentralMoment} and \ref{Prop.LowerTailMoment}. In the last inequality $C(m,A,B,\e)$ is independent of $T$ and the convergence to zero follows from our assumption that $\e \in (\frac{3}{5}, 1)$ and $I_2 \geq 1$.

Combining (\ref{Eq.JointMomentPiece1}) and (\ref{Eq.JointMomentPiece2}), we conclude (\ref{Eq.VanishS6}), which completes the proof of the theorem.

%
%
\subsection{Connection to the Gaussian free field}\label{Section6.2} In this section, we briefly review the Gaussian free field (GFF) and explain its connection to the limiting Gaussian random vector $(\xi_1,\ldots,\xi_m)$ appearing in Theorem~\ref{Thm.MomentConvergence}. For additional background on the GFF, we refer the reader to \cite{BP25, Sheffield07,WP21}. 

The GFF on the upper half-plane $\mathbb{H}$ is a random element of $(C_c^\infty( \mathbb{H} ))'$ uniquely determined by the following property: given $\varphi_1, \ldots, \varphi_m \in C_c^\infty( \mathbb{H} )$, $\left(\ang{\GFF, \varphi_1}, \ldots ,\ang{\GFF, \varphi_m}\right)$ is a zero-mean Gaussian random vector with covariance matrix
$$
\mathbb{E}\left[ \ang{\GFF, \varphi_j} \ang{\GFF, \varphi_k} \right]=\int_{\mathbb{H} \times \mathbb{H}}\left|d z_1\right|^2\left|d z_2\right|^2 \varphi_j\left(z_1\right) \varphi_k\left(z_2\right) \mathcal{G}\left(z_1, z_2\right) .
$$
Here, $|dz_i|^2 = du_i dv_i$ if $z_i = u_i + \im v_i$, and 
\begin{equation}\label{Eq.Green'sFunction}
\mathcal{G}(z, w):=-\frac{1}{2 \pi} \ln \left|\frac{z-w}{z-\bar{w}}\right|, \quad z, w \in \mathbb{H},
\end{equation}
is the Green function for the Laplace operator on $\mathbb{H}$ with Dirichlet boundary conditions. The existence of such a random element follows from \cite[Corollary 1.48 and Theorem 1.57]{BP25}.

Let $D$ be a regular domain, $f: D \rightarrow \mathbb{H}$ a diffeomorphism, and let $g(z) = (g_1(u,v), g_2(u,v))$ for $z = u + \im v$ be its inverse. We define the random element $\GFF^{f}$ of $(C_c^\infty(D))'$ by setting for $\varphi \in C^{\infty}_c(D)$
$$\ang{\GFF^f, \varphi} := \ang{\GFF, (\varphi \circ g) |J_{g}|},$$
where $J_{g}$ is the Jacobian defined for $z = u + \im v$ through
$$J_g(z) = J_{g}(u,v) = \det \begin{bmatrix} \partial_u g_1(u,v) & \partial_v g_1(u,v) \\ \partial_u g_2(u,v) & \partial_v g_2(u,v) \end{bmatrix}.$$
The random element $\GFF^{f}$ is called the $f$-pullback of the Gaussian free field on $\mathbb{H}$. The latter definition and the earlier property for $\GFF$ imply that, given $\varphi_1, \ldots, \varphi_m \in C_c^\infty(D)$, $\left(\ang{\GFF^{f}, \varphi_1}, \ldots ,\ang{\GFF^f, \varphi_m}\right)$ is a zero-mean Gaussian random vector with covariance matrix
\begin{equation}\label{Eq.Cov2}
\begin{split}
&\mathbb{E}\left[ \ang{\GFF^f, \varphi_j} \ang{\GFF^f, \varphi_k} \right]=\int_{\mathbb{H} \times \mathbb{H}}\left|d z_1\right|^2\left|d z_2\right|^2 | J_{g}(z_1)| | J_{g}(z_2)| \varphi_j\left(g(z_1)\right) \varphi_k\left(g(z_2)\right) \mathcal{G}\left(z_1, z_2\right) \\
& = \int_{D} dt_j dx_j \int_{D} dt_k dx_k \varphi_j(t_j, x_j) \varphi_k(t_k, x_k) \mathcal{G}(f(t_j,x_j), f(t_k,x_k)),
\end{split}
\end{equation}
where the last equality follows from a change of variables. 

\medskip

In our setup, we let $\crit$ be the diffeomorphism
\begin{equation}\label{Eq.CriticalPoint}
\crit(t, x) = t + \im \sqrt{-x}
\end{equation}
that maps $\mathcal{D} = \{(t,x)\in\mathbb{R}^2:x<0\}$ to the upper half-plane $\mathbb{H}$, and let $\GFFpullback$ denote the $\crit$-pullback of the Gaussian free field as above. Substituting $D = \mathcal{D}$, $f = \crit$, and (\ref{Eq.Green'sFunction}) into (\ref{Eq.Cov2}), we recognize the last line as the formula for $\mathbb{E}[\xi_j \xi_k]$ from (\ref{Eq.S1Covariance}). Consequently, $(\ang{\GFFpullback,\varphi_1},\ldots,\ang{\GFFpullback,\varphi_m}) \overset{d}{=} (\xi_1, \ldots, \xi_m)$.

%
\section*{Acknowledgements} E.D. was partially supported by Simons Foundation International through Simons Award TSM-00014004. 

\bibliographystyle{alpha}
\bibliography{reference}
\end{document}